\documentclass{amsart}

\swapnumbers \theoremstyle{plain}
\newtheorem{thm}{Theorem}[section]
\newtheorem{lem}[thm]{Lemma}

\newtheorem{lem-defn}[thm]{Lemma and definition}

\newtheorem{cor}[thm]{Corollary}
\newtheorem{prop}[thm]{Proposition}
\theoremstyle{definition}\newtheorem{rem}[thm]{Remark}
\newtheorem{defn}[thm]{Definition}

\theoremstyle{definition}

\newcommand{\Cal}{\mathcal}

\newcommand{\R}{\mathbb{R}}

\newcommand{\C}{\mathbb{C}}
\newcommand{\D}{\mathbb{D}}
\newcommand{\N}{\mathbb{N}}

\DeclareMathOperator{\supp}{supp} \DeclareMathOperator{\im}{Im}
\DeclareMathOperator{\ke}{Ker} \DeclareMathOperator{\rea}{Re}

\numberwithin{equation}{section}

\usepackage{soul,cancel,url}
\usepackage{amsfonts}
\usepackage{hyperref}

\usepackage{color}
\usepackage{amsmath}
    \usepackage{amsxtra}
    \usepackage{amstext}
    \usepackage{amssymb}
   \usepackage{latexsym}
   \usepackage{graphicx}

\title {Similarity of holomorphic matrices on 1-dimensional Stein spaces}
\author{J\"urgen Leiterer}
\address{Institut f\"ur Mathematik \\
Humboldt-Universit\"at zu Berlin \\Rudower Chaussee 25\\D-12489 Berlin , Germany}
\email{leiterer@mathematik.hu-berlin.de}

\thanks{MSC 2010: 32E10, 47A56, 15A21.}

\thanks{Keywords: holomorphic matrices, similarity, Oka principle, 1-dimensional Stein spaces}

\thanks{This is a revised version of a part of a preprint, sent to some colleagues in September 2016 and then posted in the arXiv \cite{Le0}.}

\date{}

\dedicatory{Dedicated to the memory of Selim Krein}

\begin{document}

\begin{abstract}  R. Guralnick [Linear Algebra Appl. 99, 85-96 (1988)] proved that two holomorphic matrices on a noncompact connected Riemann surface, which are locally holomorphically similar, are globally holomorphically similar. In the preprints [arXiv:1703.09524] and [arXiv:1703.09530],  a generalization of this to arbitrary (possibly, nonsmooth) 1-dimensional Stein spaces was obtained. The  present  paper contains a revised version of the proof from [arXiv:1703.09524].
The method of this revised proof can be used also in the higher dimensional case, which will be the subject of a forthcoming paper.
\end{abstract}

\maketitle

\section{Introduction}

Let $X$ be a (reduced) complex space, e.g., a complex manifold or an analytic subset of a complex manifold. Let  $\mathrm{Mat}(n\times n,\C)$ be the algebra of complex $n\times n$ matrices, and  $\mathrm{GL}(n,\C)$ the group of  invertible complex $n\times n$ matrices.

Two holomorphic maps $A,B:X\to \mathrm{Mat}(n\times n,\C)$ are called
(globally) {\bf  holomorphically  similar on $ X$} if there is a holomorphic  map $H:X\to\mathrm{GL}(n,\C)$ with $B=H^{-1}AH$ on $X$.
They are called
{\bf locally holomorphically similar at $\xi\in X$}  if there is a neighborhood $U$ of $\xi$ such that $A\vert_U$ and $B\vert_U$ are holomorphically similar on $U$.
Correspondingly,   {\bf continuous} and {\bf $\Cal C^k$ similarity} are defined.

\vspace{2mm}
R. Guralnick \cite{Gu} proved the following theorem.

\begin{thm}\label{24.10.16+} Suppose  $X$ is a noncompact connected Riemann surface. Then any two holomorphic maps $A,B:X\to  \mathrm{Mat}(n\times n,\C)$, which are locally holomorphically similar at each point of $X$,  are globally holomorphically similar on $X$.
\end{thm}

In \cite{Le0}, the following generalization was obtained.
\begin{thm}\label{18.10.16} The claim of Theorem 1.1 remains true if  $X$ is an arbitrary  1-dimensional Stein space (for example, a 1-dimensional closed analytic subset of some $\C^N$, or, more general, of a domain of holomorphy in $\C^N$).
\end{thm}

Guralnick's proof of Theorem \ref{24.10.16+} consists in proving a theorem for matrices with elements in a Bezout ring (with some extra properties) and then  applying this to the ring of holomorphic functions on a noncompact connected Riemann surface. The ring of holomorphic functions on an arbitrary (possibly nonsmooth) 1-dimensional Stein space need not be  Bezout. Therefore, it seems  that this proof cannot be used to  prove Theorem \ref{18.10.16}, at least not in a straightforward way.

The proof of Theorem \ref{18.10.16} given in \cite{Le0} proceeds as follows.
First we use the  Oka principle for Oka-pairs of Forster and Ramspott \cite{FR1} to show that Theorem \ref{18.10.16} is equivalent to a certain topological statement (see Theorem \ref{4.10.17'neu} below). Then we prove this topological statement. This proof is independent of Guralnick's result. Using Guralnick's result, this proof can be shortened, passing to the normalization of $X$ (which is smooth). This is shown in \cite{Le2}.

The aim of the present paper is to give a revised version of the proof  from \cite{Le0}.
The method of this revised version is useful also if $X$ higher dimensional. For example, in a forthcoming paper, we will show that the claim of Theorem \ref{24.10.16+} remains valid if $X$ is a convex domain in $\C^2$ (for convex domains in $\C^3$, this is not true). With this in view,  already in the present paper, sometimes $X$ is allowed to be of arbitrary dimension.

{\bf Note added in proof.} Another proof of Theorem \ref{18.10.16} is outlined by F. Forstneri\v{c} in  the second edition of his  book \cite[Theorem 6.14.9]{Fc}. He does not use the above mentioned Oka principle of Forster and Ramspott, but a new Oka principle  originally established by him in \cite{Fc1}.

\section{Notations and our use of the language of sheaves}\label{20.6.16}

$\N$ is the set of natural numbers including $0$. $\N^*=\N\setminus\{0\}$.

If $n,m\in \N^*$, then $\mathrm{Mat}(n\times m,\C)$ is the space of complex $n\times m$ matrices ($n$ rows, $m$ columns), and
$\mathrm{GL}(n,\C)$ is the group of invertible complex $n\times n$ matrices.

The unit matrix in $\mathrm{Mat}(n\times n,\C)$ will be denoted by $I_n$ or simply by $I$.

If a matrix  $\Phi\in \mathrm{Mat}(n\times m,\C)$ is interpreted as a linear map from $\C^m$ to  $\C^n$, then $\ke \Phi$ denotes the kernel,  $\im \Phi$ the image and $\Vert\Phi\Vert$ the operator norm  (induced by the Euclidean norm) of  $\Phi$.

By a complex space we always mean a {\em reduced} complex space in the sense of \cite{GR}, which is the same as an {\em analytic space} in the terminology used in \cite{C} and \cite{L}.

From now on, in this section,  $X$ is  topological  space,  and  $G$ is a topological group (possibly non-abelian).

By $\Cal C^G$, or more precisely by $\Cal C^G_X$, we denote the sheaf of continuous $G$-valued maps on $X$, that is, the map which assigns to each nonempty open $U\subseteq X$ the group
 $ \Cal C^{G}(U)$  of all continuous maps $f:U\to {G}$. Also $\mathcal C^G(\emptyset):=\{1\}$ ($1$ being the neutral element of $G$).

 A subsheaf  of $\Cal C_X^G$ is a map $\Cal F$ which assigns to each open $U\subseteq X$ a subgroup $\Cal F(U)$ of $\Cal C^G(U)$ such that:
\begin{itemize}
\item[--] If $V\subseteq U$ are nonempty open subsets of $X$, then, for each $f\in \Cal F(U)$, the restriction of $f$ to $V$, $f\vert_V$, belongs to $\Cal F(V)$.
\item[--] If $U\subseteq X$ is  open  and $f\in \Cal C^G(U)$ is  such that, for each $\xi\in U$, there is an
open neighborhood $V\subseteq U$ of $\xi$ with $f\vert_V\in \Cal F(V)$, then  $f\in \Cal F(U)$.
\end{itemize}The elements of $\Cal F(U)$ are called {\bf sections} of $\mathcal F$ over $U$.

 If $\Cal F$ and $\Cal G$ are two subsheaves of $\Cal C^{G}_X$, then $\Cal F$ is called a subsheaf of $\Cal G$ if $\Cal F(U)\subseteq \Cal G(U)$ for all open $U\subseteq X$.

  If $X$ is  a complex space and  $G$ a complex Lie group, then we denote by $\Cal O^{G}_X$, or simply by $\Cal O^{G}$,  the  subsheaf of $\Cal C^G_X$ which assigns to each nonempty open $U\subseteq X$, the group $\Cal O^G(U)$ of all holomorphic maps from $U$ to $G$.

Let $\Cal F$ be a subsheaf of $\Cal C^G_X$.

Let
 $\Cal U=\{U_i\}_{i\in I}$  an open covering of $X$.

A family $f_{ij}\in\Cal F(U_i\cap U_j)$, $i,j\in I$, is called a {\bf $(\Cal U,\Cal F)$-cocycle}, or simply a {\bf cocycle}, if (the group operation of $G$ being a multiplication)
\begin{equation*}
f_{ij}f_{jk}=f_{ik}\quad\text{on}\quad U_i\cap U_j\cap U_k\quad\text{for all}\quad i,j,k\in I.\quad\footnote{We use the  convention that  statements like
``$f=g$ on $\emptyset$'' or ``$f:=g$ on $\emptyset$''
have to be omitted.}
\end{equation*}Note that then always $f^{-1}_{ij}=f^{}_{ji}$ and $f^{}_{ii}\equiv 1$.
The set of all $(\Cal U,\Cal F)$-cocycles will be denoted by $Z^1(\Cal U,\Cal F)$.

Let  $f=\{f_{ij}\}\in Z^1(\Cal U,\Cal F)$. We say that $f$ {\bf splits}  (or {\bf is trivial}) if there exists a family $h_i\in \Cal F(U_i)$, $i\in I$, such that
\[
f^{}_{ij}=h^{}_i h^{-1}_j\quad\text{on}\quad U_i\cap U_j\quad\text{for all}\quad i,j\in I.
\]

We say that $f$ is an {\bf $\Cal F$-cocycle} on $X$, if there exists an open covering $\Cal U$ of $X$ with $f\in Z^1(\Cal U,\Cal F)$. Then $\Cal U$ is called the {\bf covering of} $f$.
As usual, we write $H^1(X,\Cal F)=0$
to say that each $\Cal F$-cocycle on $X$ splits, and $H^1(X,\Cal F)\not=0$ to say that there exist non-splitting $\Cal F$-cocycles on $X$.

Now let $\Cal U^*=\{U^*_\alpha\}_{\alpha\in I^*}$ be a second open covering of $X$, which is a refinement of $\Cal U$, i.e., there is a map
$\tau:I^*\to I$ with $U^*_\alpha\subseteq U_{\tau(\alpha)}$ for all $\alpha\in I^*$.
Then we say  that
a $(\Cal U^*,\Cal F)$-cocycle $\{f^*_{\alpha\beta}\}^{}_{\alpha,\beta\in I^*}$  is {\bf induced} by a $(\Cal U,\Cal F)$-cocycle $\{f_{ij}\}_{i,j\in I}$
if this map $\tau$ can be chosen so that
\[
f^*_{\alpha\beta}=f_{\tau(\alpha)\tau(\beta)}^{}\quad\text{on}\quad U^*_i\cap U^*_j\quad\text{for all}\quad \alpha,\beta\in I^*.
\]
We need the following  well-known  and simple proposition, see \cite[p. 41]{Hi}  for ``only if'' and  \cite[p. 101]{C} for ``if''.
\begin{prop}\label{17.12.15--} Let $f\in Z^1(\Cal U,\Cal F)$ and $f^*\in Z^1(\Cal U^*,\Cal F)$ such that $f^*$ is induced by $f$.
Then $f$ splits if and only if $f^*$ splits.
\end{prop}

Let $Y$ be a nonempty open subset of $X$.

Then we denote by $\Cal F\vert_Y$ the subsheaf of $\Cal C^{G}_Y$ defined by $\Cal F\vert_Y(U)=\Cal F(U)$ for each open $U\subseteq Y$. $\Cal F\vert_Y$ is called the {\bf restriction} of $\Cal F$ to $Y$.

If $\Cal U=\{U_i\}_{i\in I}$ is an open covering of $X$, then we define $\Cal U\cap Y=\big\{U_i\cap Y\}_{i\in I}$, and, for each $f=\{f_{ij}\}_{i,j\in I}\in Z^1(\Cal U,\Cal F)$,  we denote by $f\vert_Y=\{(f\vert_Y)^{}_{ij}\}_{i,j\in I}$ the $(\Cal U\cap Y,\Cal F\vert_Y)$-cocycle defined by
\[
(f\vert_Y)^{}_{ij}=f^{}_{ij}\big\vert_{U_i\cap U_j\cap Y}\quad\text{for}\quad i,j\in I.
\] We call $f\vert_Y$ the {\bf restriction} of $f$ to $Y$.

\begin{rem}\label{30.9.17}  Let $\Cal U=\{U_i\}_{i\in I}$ be an open covering of $X$, and $f=\{f_{ij}\}_{i,j\in I}\in Z^1(\Cal U,\Cal F)$. Then $f\vert_{U_i}$ splits for each $i\in I$ with $U_i\not=\emptyset$.

Indeed, the one-set family $\{U_i\}$ is an open  covering of $U_i$ which is a refinement of $\mathcal U\cap U_i$, and there is precisely one $(\{U_i\},\mathcal F)$-cocycle which is induced by $f\vert_{U_i}$, namely $\{f_{ii}\}$. Since  $f_{ii}\equiv 1$, it is trivial that $\{f_{ii}\}$ splits. Therefore it follows from Proposition \ref{17.12.15--} that $f\vert_{U_i}$ splits.
\end{rem}

\section{A topological condition for global holomorphic similarity}\label{4.10.17}

\begin{defn}\label{26.9.17n}
Let  $\Phi\in\mathrm{Mat}(n\times n,\C)$. Then we denote by
 $\mathrm{Com\,} \Phi$ the algebra of all $\Theta\in \mathrm{Mat}(n\times n,\C)$ with $\Phi \Theta=\Theta \Phi$, and by $\mathrm{GCom\,} \Phi$ we denote the group of invertible elements of $\mathrm{Com\,} \Phi$.
It is easy to see that
\begin{align}&\label{17.8.16+}\mathrm{GCom\,} \Phi=\mathrm{GL}(n,\C)\cap \mathrm{Com\,} \Phi,\\
&\label{26.8.16+}
 \mathrm{Com\,}(\Gamma^{-1}\Phi\Gamma)=\Gamma^{-1}(\mathrm{Com \,}\Phi)\Gamma\quad\text{for all}\quad\Gamma\in\mathrm{GL}(n,\C).
 \end{align}

Now let $X$ be a complex space (of arbitrary dimension), and $A:X\to\mathrm{Mat}(n\times n,\C)$ a holomorphic map.

We introduce the families
\[
 \mathrm{Com\,}A:=\big\{\mathrm{Com\,}A(\zeta)\big\}_{\zeta\in X}\quad\text{and}\quad \mathrm{GCom\,}A:=\big\{\mathrm{GCom\,}A(\zeta)\big\}_{\zeta\in X}.
\]
If the dimension of  $\mathrm{Com\,}A(\zeta)$ does not depend on $\zeta$, then it is well-known (and easy to see) that $\mathrm{Com\,}A$ is a holomorphic vector bundle. But also in this special case, the family of groups $\mathrm{GCom\,}A$ need not be locally trivial. It is possible that $\mathrm{Com\,}A$ is a holomorphic vector bundle, but $\mathrm{GCom\,}A$ is even not locally trivial as a family of topological spaces.  For an example, see \cite[Sec. 4]{Le2}.

Nevertheless the sheaves of holomorphic and  continuous sections of $\mathrm{Com\,}A$ and $\mathrm{GCom\,}A$ are well-defined. We denote them by $\Cal O^{\mathrm{Com\,}A}$, $\Cal O^{\mathrm{GCom\,}A}$, $\Cal C^{\mathrm{Com\,}A}$ and $\Cal C^{\mathrm{GCom\,}A}$, respectively.

We define a subsheaf $\widehat{\Cal O}^{\mathrm{Com\,}A}$  of $\Cal C^{\mathrm{Com\,}A}$  as follows: if  $U$ is a nonempty open subset of $X$, then   $\widehat{\Cal O}^{\mathrm{Com\,}A}(U)$ is the subalgebra of all $f\in\Cal C^{\mathrm{Com\,}A}(U)$
such that, for each $\xi\in U$,
\begin{equation}\label{19.8.16}\begin{cases}&\text{there exist a neighborhood }V_\xi\text{ of }\xi\\&\text{and }h_\xi\in \Cal O^{\mathrm{Com\,}A}(V_\xi)\text{ with }h(\xi)=f(\xi).
\end{cases}\end{equation}Further, we define a subsheaf $\widehat{\Cal O}^{\mathrm{GCom\,}A}$ of $\Cal C^{\mathrm{GCom\,}A}$ setting  $\widehat{\Cal O}^{\mathrm{GCom\,}A}(U)=\Cal C^{\mathrm{GL}(n,\C)}(U)\cap\widehat{\Cal O}^{\mathrm{Com\,}A}(U)$ for each nonempty open $U\subseteq X$.
\end{defn}

The Oka principle for Oka pairs of Forster and Ramspott \cite[Satz 1]{FR1}, yields the following

\begin{prop}\label{4.10.17'} Let $X$ be a Stein space, and let $A:X\to \mathrm{Mat}(n\times n,\C)$ be holomorphic. Then  each $\mathcal O^{\mathrm{GCom\,}A}$-cocycle, which splits as an $\widehat{\mathcal O}^{\mathrm{GCom\,}A}$-cocycle, splits also as an $\mathcal O^{\mathrm{GCom\,}A}$-cocycle.
\end{prop} Indeed,
 it is easy to see that, for each nonempty open $U\subseteq X$ we have: If $h\in \Cal O^{\mathrm{Com\,}A}(U)$, then $e^h\in\Cal O^{\mathrm{GCom\,}A}(U)$, and, if $ H\in\Cal O^{\mathrm{GCom\,}A}(U)$  with $\sup_{\zeta\in U}\Vert H(\zeta)-I\Vert<1 $, then
\[
\log H:=-\sum_{\mu=1}^\infty \frac{(H-I)^\mu}{\mu}
\]belongs to $\Cal O^{\mathrm{Com\,}A}(U)$.
This shows that $\Cal O^{\mathrm{GCom\,}A}$ is a {\em coherent $\Cal O$-subsheaf} of $\Cal O^{\mathrm{GL}(n,\C)}$ in the sense of \cite[\S 2]{FR1}, where $\Cal O^{\mathrm{Com\,}A}$ is the {\em generating sheaf of Lie algebras}.
Moreover, as observed in \cite[\S 2.3, example b)]{FR1}), the pair $\big(\Cal O^{\mathrm{GCom\,}A},\widehat{\Cal O}^{\mathrm{GCom\,}A}\big)$ is an {\em admissible pair} in the sense of \cite{FR1}, which, trivially, satisfies condition (PH) in Satz 1 of \cite{FR1}). Therefore Proposition \ref{4.10.17'} is one of the statements of that Satz 1.

\begin{thm}\label{4.10.17'neu}
 Let $X$ be a Stein space, and let $A:X\to \mathrm{Mat}(n\times n,\C)$ be holomorphic. Then the following are equivalent.

 {\em (i)} Each holomorphic  $B:X\to \mathrm{Mat}(n\times n,\C)$, which is locally holomorphically similar to $A$ at each point of $X$,
is globally holomorphically similar to $A$ on $X$.

{\em (ii)} Each $\mathcal O^{\mathrm{GCom\,}A}$-cocycle on $X$, which splits as a $\mathcal C^{\mathrm{GL\,}(n,\C)}$-cocycle, splits also as an $\widehat{\mathcal O}^{\mathrm{GCom\,}A}$-cocycle.
\end{thm}

\begin{proof} (ii) $\Longrightarrow$ (i):
Let $B:X\to \mathrm{Mat}(n\times n,\C)$ be holomorphic and locally holomorphically similar to $A$ at each point of $X$. Then we can find an open covering $\{U_i\}_{i\in I}$ of $X$ and holomorphic  maps $H_i:U_i\to\mathrm{GL}(n,\C)$, $i\in I$, such that
\begin{equation}\label{28.6.16}B=H_i^{-1}AH_i^{}\quad\text{on}\quad U_i.
\end{equation} Hence,  for all $i,j\in I$ with $U_i\cap U_j\not=\emptyset$,  $AH_i^{}H_j^{-1}=H_i^{}H_j^{-1}A$ on $U_i\cap U_j$, i.e., $H_i^{}H_j^{-1}\in\Cal O^{\mathrm{GCom\,}A}(U_i\cap U_j)$. Clearly,
\[
(H_i^{}H_j^{-1})(H_j^{}H_k^{-1})=H_i^{}H_k^{-1}\quad\text{on}\quad U_i\cap U_j\cap U_k,\qquad i,j,k\in I.
\]
Therefore, the family $\big\{H_i^{}H_j^{-1}\big\}_{i,j\in I}$ is a well-defined $O^{\mathrm{GCom\,}A}$-cocycle. It is clear that this cocycle splits as an $\mathcal O^{\mathrm{GL\,}(n,\C)}$-cocycle. In particular it splits as a $\mathcal C^{\mathrm{GL\,}(n,\C)}$-cocycle. Since condition (ii) is satisfied, this implies that $\{h_{ij}\}_{i,j\in I}$ splits as an $\widehat{\mathcal O}^{\mathrm{GCom\,}A}$-cocycle. By Proposition \ref{4.10.17'}, this further implies that $\{h_{ij}\}_{i,j\in I}$ splits as an $\mathcal O^{\mathrm{GCom\,}A}$-cocycle, i.e., there is a family $h_i\in\Cal O^{\mathrm{GCom\,}A}(U_i)$ with $H_i^{}H_j^{-1}=h_i^{}h_j^{-1}$ on $U_i\cap U_j$. Therefore $h_i^{-1}H_i^{}=h_j^{-1}H_j^{}$ on $U_i\cap U_j$, and we have
 a well-defined global holomorphic  map $H:X\to\mathrm{GL}(n,\C)$ with $H=h_i^{-1}H_i^{}$ on $U_i$ for all $i\in I$.
 By  \eqref{28.6.16} and since $h_i^{}Ah_i^{-1}=A$, we get $H^{-1}AH=B$ on $X$.

 (i) $\Longrightarrow$ (ii): Let an open covering $\mathcal U=\{U_i\}_{i\in I}$ of $X$ and a $(\mathcal U,\mathcal O^{\mathrm{GCom\,}A})$-cocycle $f=\{f_{ij}\}_{i,j\in I}$ be given such that $f$ splits  as a $\mathcal C^{\mathrm{GL\,}(n,\C)}$-cocycle.
Then, by Grauert's Oka principle \cite[Satz I]{Gr} (see also \cite[Theorem 5.3.1]{Fc}), $f$ even splits as an $\mathcal O^{\mathrm{GL\,}(n,\C)}$-cocycle. This means (we may assume that $U_i\not=\emptyset$ for all $i$) that there exists a family  of holomorphic maps $f_i:U_i\to \mathrm{GL\,}(n,\C)$ such that
\begin{equation}\label{4.10.17--}
f_{ij}^{}=f_i^{}f_j^{-1}\quad\text{on}\quad U_i\cap U_j,\quad i,j\in I.
\end{equation} Since $f_{ij}\in \mathcal O^{\mathrm{GCom\,}A}(U_i\cap U_j)$, this implies that
\[
f_i^{}f_j^{-1}A=Af_i^{}f_j^{-1}\quad\text{and, hence,}\quad f_j^{-1}A f_j^{}=f_i^{-1}Af_i^{}\quad\text{on}\quad U_i\cap U_j,\quad i,j\in I.
\]Hence, there is a well-defined holomorphic map $B:X\to \mathrm{Mat}(n\times n,\C)$ with
\begin{equation}\label{4.10.17---}
B=f_i^{-1}Af_i^{}\quad\text{on}\quad U_i,\quad i\in I.
\end{equation}
From  its definition it is clear that $B$ is locally holomorphically similar to $A$ at each point of $X$. Since condition (i) is satisfied, it follows  that  there is a holomorphic map $T:X\to \mathrm{GL\,}(n,\C)$ such that
\begin{equation}\label{4.10.17+}B=T^{-1}AT\quad\text{on}\quad X.
\end{equation} Set $h^{}_i=f_i^{}T^{-1}_{}$ on $U_i$. By \eqref{4.10.17+} and \eqref{4.10.17---}, then, on each $U_i$,
\[
h^{}_i A=f_i^{}T^{-1}_{}A=f_i^{}BT^{-1}=f_i^{}f_i^{-1}Af_i^{}T^{-1}=Af_i^{}T^{-1}=Ah^{}_i,
\] i.e., $h_i\in \mathcal O^{\mathrm{GCom\,}A}(U_i)$. Moreover, by \eqref{4.10.17--},
\[
h_i^{}h_j^{-1}=f_i^{}T^{-1}_{}Tf_j^{-1}=f_i^{}f_j^{-1}=f_{ij}^{}\quad\text{on}\quad U_i\cap U_j.
\]So, $f$ splits as an $\mathcal O^{\mathrm{GCom\,}A}$-cocycle and, above all, as an $\widehat{\mathcal O}^{\mathrm{GCom\,}A}$-cocycle.
\end{proof}

For us, the following immediate corollary of Theorem \ref{4.10.17'neu} is important.

\begin{cor}\label{4.10.17*} Let $X$ be a Stein space, and let $A,B:X\to \mathrm{Mat}(n\times n,\C)$ be holomorphic maps, which are locally holomorphically similar at each point of $X$. If
\[H^1(X,\widehat{\mathcal O}^{\mathrm{GCom\,}A})=0,\] then $A$ and $B$ are globally holomorphically similar on $X$.
\end{cor}

\section{Bumps on Riemann surfaces}\label{23.7.16+}

\begin{defn}\label{3.8.16}Let $X$ be a Riemann surface. Denote by $\Delta$ the closed unit disk centered at the origin  in $\C$, and set
\begin{align*}
&\Delta_I=\big\{u\in\Delta\;\big\vert\; \im u\le 0\text{ and } 1/2\le \vert u\vert\le 1\big\},\\
&\Delta_{II}=\big\{u\in\Delta\;\big\vert\;\vert\im u\vert\le \vert \rea u\vert\text{ and }1/2\le \vert u\vert\le 1\big\}.
\end{align*}
A pair $(B_1,B_2) $ will be called a {\bf bump in $X$} if $B_1,B_2$ are  compact subsets of $X$ such that either
\begin{equation}B_1\cap B_2=\emptyset,
\end{equation} or there exist an open neighborhood $U$ of $B_2$ and a $\mathcal C^\infty$-diffeomorphism, $z$, from $U$ onto an open neighborhood of $\Delta$ such that\footnote{$\big\{\vert z\vert\le 1\big\}:=\{z\in\Delta\}:=z^{-1}(\Delta):=\big\{\zeta\in U\,\big\vert\,\vert z(\zeta)\vert\le 1\big\}$ etc.}
\begin{align}&\label{23.9.17'''}B_2\subseteq \{\vert z\vert\le 1\},\\
&\label{30.9.17a}B_1\cap \{\vert z\vert\le 1\}\subseteq B_2,
\end{align} and one of the following two conditions is satisfied.
\begin{align}&\label{23.9.17}
B_1\cap\big\{1/2\le \vert z\vert\le 1\big\}=B_2\cap\big\{1/2\le \vert z\vert\le 1\big\}=\{z\in \Delta_I),\\
&\label{23.9.17'}B_1\cap\big\{1/2\le \vert z\vert\le 1\big\}=B_2\cap\big\{1/2\le \vert z\vert\le 1\big\}=\{z\in\Delta_{II}\}.
\end{align}
An $m$-tuple $(B_1,\ldots,B_m)$, $m\ge 2$, of compact subsets of $X$ is called a {\bf bump extension in $X$} if, for each $1\le \mu\le m-1$, $\big(B_1\cup\ldots\cup B_\mu, B_{\mu+1}\big)$ is a bump in $X$.
\end{defn}

\begin{lem}\label{7.8.16} Let $X$ be a Riemann surface, and let $\rho:X\to \R$ be a $\Cal C^\infty$ function such that, for some real numbers $\alpha<\beta$, the set $\{\rho\le \beta\}$ is compact and $\rho$ has no critical points on $\{\alpha\le \rho\le \beta\}$. Moreover, let  $\Cal U$ be an open covering of $X$.
Then there exist $B_1,\ldots,B_{m}$ such that
\begin{itemize}
\item[(i)] $\big(\{\rho\le \alpha\},B_1,\ldots B_m\big)$ is a bump extension in $X$;
\item[(ii)] $\{\rho\le \alpha\}\cup B_1\cup\ldots\cup B_m=\{\rho\le \beta\}$;
\item[(iii)] for each $1\le \mu\le m$, $\{\rho\le \alpha-1\}\cap B_\mu=\emptyset$ and  $B_\mu$ is contained in at least one set of $\Cal U$.
\end{itemize}
\end{lem}
\begin{proof} It is sufficient  to prove that, for each $\alpha\le t\le \beta$,
 there exists $\varepsilon>0$ such that, if $t-\varepsilon\le t_1\le t\le t_2\le t+\varepsilon$, then there exist $B_1,\ldots,B_m$ such that
\begin{itemize}
\item[(i')] $\big(\{\rho\le t_1\},B_1,\ldots,B_m\big)$ is a bump extension in $X$;
\item[(ii')] $\{\rho\le t_1\}\cup B_1\cup\ldots\cup B_m=\{\rho\le t_2\}$;
\item[(iii')] for each $2\le \mu\le m$, $\{\rho\le \alpha-1\}\cap B_\mu=\emptyset$ and $B_\mu$ is contained in at least one  set of $\Cal U$.
\end{itemize}

Let $\alpha\le t\le \beta$ be given.

Since $\rho$ has no critical points on $\{\rho=t\}$ and $\{\rho=t\}$ is compact, then we can find $\widetilde \varepsilon>0$, open subsets $\widetilde U_1,\ldots,\widetilde U_m$ of $X$, and $\mathcal C^\infty$ diffeomorphisms $\widetilde z_\mu$ from $\widetilde U_\mu$ onto a neighborhood of $\Delta$  such that
\begin{itemize}
\item[(a)]  $\big\{t-\widetilde\varepsilon\le \rho\le t+\widetilde \varepsilon\}\subseteq \big\{\vert\widetilde z_1\vert<1/8\}\cup\ldots\cup \{\vert\widetilde z_m\vert<1/8\}$;
\item[(b)]   $\rho=\im\widetilde z_\mu+t$ on $\widetilde U_\mu$ for  $1\le \mu\le m$;
\item[(c)] for each $1\le \mu\le m$, $\widetilde U_\mu\cap\{\rho\le\alpha-1\}=\emptyset$  and $\widetilde U_\mu$ is contained  in at least one  set of  $\Cal U$.
\end{itemize}
By (a) we can choose  $\Cal C^\infty$-functions $\chi^{}_1,\ldots,\chi^{}_m:X\to [0,1]$ such that
\begin{itemize}
\item[(d)]$\supp\chi^{}_\mu\subseteq \{\vert\widetilde  z_\mu\vert<1/4\}$ for $1\le \mu\le m$;
\item[(e)]$\sum_{\mu=1}^m\chi^{}_\mu=1$ on $\big\{t-\widetilde \varepsilon\le \rho\le t+\widetilde \varepsilon\}$, and $0\le\sum_{\mu=1}^m\chi^{}_\mu\le1$  on $X$.
\end{itemize}
Further, for each $1\le\mu\le m$, take an open set $U_\mu\subseteq\widetilde U_\mu$, which is relatively compact in $\widetilde  U_\mu$ and such that $\widetilde z_\mu(U_\mu)$ is still a neighborhood of $\Delta$. Then we can find $\varepsilon>0$ so small that, for all $v,w\in\C$ with $\vert v\vert,\vert w\vert\le 2\varepsilon$, and, for  $1\le \mu\le m$,
\begin{itemize}
\item[(f)]  the function
$\widetilde z_\mu+v+w\sum_{\nu=1}^\mu\chi_\nu$ restricted to $U_\mu$ is  a $\mathcal C^\infty$ diffeomorphism from $U_\mu$ onto a neighborhood of $\Delta$;
\item[(g)] $\big\{\vert\widetilde z_\mu\vert\le 1/4\big\}\subseteq\big\{\zeta\in U_\mu\;\big\vert\;\vert\widetilde z_\mu(\zeta)+v+w\sum_{\nu=1}^\mu\chi_\nu(\zeta)\vert\le 1/2\big\}$.
\end{itemize}
Moreover, we may assume that
\begin{itemize}
\item[(h)] $\varepsilon<\widetilde\varepsilon/4$.\end{itemize}

To  prove that this $\varepsilon$ has the required property, let $t_1, t_2$ with
$t-\varepsilon\le t_1\le t\le t_2\le t+\varepsilon$ be given.
Define
\begin{itemize}
\item[] on $U_\mu$: $z_\mu=\widetilde z_\mu+i(t-t_1)-i(t_2-t_1)\sum_{\nu=1}^{\mu}\chi_\nu$ for $1\le\mu\le m$;
\item[] on $X$: $\rho_0=\rho-t_1$ and $\rho_\mu=\rho-t_1-(t_2-t_1)\sum_{\nu=1}^\mu\chi_\nu$  for $1\le \mu\le m$;
\item[] $B_\mu=\big\{\rho_{\mu}\le 0\big\}\cap\big\{\vert z_\mu\vert\le 1\big\}$ for $1\le \mu\le m$.
\end{itemize}
Then, by (e) and (h), $\sum_{\nu=1}^m\chi^{}_\nu=1$ on $\big\{t-4\varepsilon\le \rho\le t+4 \varepsilon\}$, and $0\le\sum_{\nu=1}^m\chi^{}_\mu\le1$ everywhere  on $X$. Since $0\le t_2-t_1\le 2\varepsilon$ and $0\le t-t_1\le \varepsilon$, it follows that
\[\rho_m\begin{cases}=\rho-t_2\quad&\text{on}\quad\big\{t-4\varepsilon\le \rho\le t+4\varepsilon\big\},\\
\ge\rho-t_1-(t_2-t_1)>0\quad&\text{on}\quad\big\{\rho\ge t+4\varepsilon\big\},\\
\le \rho-t_1<0\quad&\text{on} \quad \big\{\rho\le t-4\varepsilon\big\}.
\end{cases}
\]
Therefore
\begin{equation}\label{12.7.17'}
\big\{\rho_m\le 0\big\}=\big\{\rho\le t_2\big\}.
\end{equation}
By (g), $\big\{\vert\widetilde z_\mu\vert\le 1/4\big\}\subseteq \big\{\vert z_\mu\vert\le 1/2\big\}$ for $1\le \mu\le m$. Together with (d), this gives
\begin{equation}\label{6.9.17}
 \rho_\mu=\rho_{\mu-1}\quad\text{on}\quad X\setminus\big\{\vert z_\mu\vert\le 1/2\big\}\quad\text{for}\quad 1\le \mu\le m.
 \end{equation}

Let $1\le \mu\le m$. Then, by \eqref{6.9.17},
\[
\{\rho_{\mu-1}\le 0\}=\Big(\{\rho_{\mu-1}\le 0\}\cap \{\vert z_\mu\vert\le 1\}\Big)
\cup\Big(\{\rho_{\mu}\le 0\}\cap \big(X\setminus\{\vert z_\mu\vert\le 1\}\big)\Big)
\]and, further,
\[\begin{split}
\{\rho_{\mu-1}\le 0\}\cup B_\mu=&\Big(\{\rho_{\mu-1}\le 0\}\cap \{\vert z_\mu\vert\le 1\}\Big)\\
&\cup\Big(\{\rho_{\mu}\le 0\}\cap \big(X\setminus\{\vert z_\mu\vert\le 1\}\big)\Big)\cup \Big(\{\rho_{\mu}\le 0\}\cap \{\vert z_\mu\vert\le 1\}\Big)\\
=&\Big(\{\rho_{\mu-1}\le 0\}\cap \{\vert z_\mu\vert\le 1\}\Big)\cup\{\rho_\mu\le 0\}.
\end{split}\]Since $\rho_{\mu-1}\ge \rho_\mu$ and therefore  $\{\rho_{\mu-1}\le 0\}\cap \{\vert z_\mu\vert\le 1\}\subseteq \{\rho_{\mu}\le 0\}$, it follows that
\begin{equation*}
\{\rho_{\mu-1}\le 0\}\cup B_\mu=\{\rho_\mu\le 0\},
\end{equation*}and, hence,
\begin{equation*}
\{\rho_0\le 0\}\cup B_1\cup\ldots\cup B_\mu=\{\rho_\mu\le 0\}.
\end{equation*} Since   $\rho_0=\rho-t_1$, so we have proved that
\begin{equation}\label{6.9.17n}
\{\rho\le t_1\}\cup B_1\cup\ldots\cup B_\mu=\{\rho_\mu\le 0\}\quad\text{for}\quad 1\le \mu\le m.
\end{equation}

Since $\rho_1\ge\rho_0=\rho-t_1$, we have $\{\rho\le t_1\}=\{\rho_0\le 0\}\subseteq\{\rho_1\le 0\}$, which implies by definition of $B_1$ that
\begin{equation}\label{1.10.17}
\{\rho\le t_1\}\cap\{\vert z_1\vert\le 1\}\subseteq B_1.
\end{equation}  For $2\le \mu\le m$, it follows from \eqref{6.9.17n} that
\[
\Big(\{\rho\le t_1\}\cup B_1\cup\ldots\cup B_{\mu-1}\Big)\cap \{\vert z_\mu\vert\le 1\}=\{\rho_{\mu-1}\le 0\}\cap\{\vert z_\mu\vert\le 1\}.
\] Sine $\rho_{\mu-1}\le \rho_\mu$, this implies by definition of $B_\mu$ that
\begin{equation}\label{1.10.17'}
\Big(\{\rho\le t_1\}\cup B_1\cup\ldots\cup B_{\mu-1}\Big)\cap \{\vert z_\mu\vert\le 1\}\subseteq B_\mu\quad\text{for}\quad 2\le \mu\le m.
\end{equation}

Further it follows from \eqref{6.9.17} that
\begin{multline}\label{15.9.17'} \{\rho_{\mu-1}\le 0\}\cap \big\{1/2\le \vert z_\mu\vert\le 1\big\}=\{\rho_\mu\le 0\}\cap\big\{1/2\le \vert z_\mu\vert\le 1\big\}\\=B_\mu\cap \big\{1/2\le \vert z_\mu\vert\le 1\big\}\quad\text{for}\quad 1\le \mu\le m.
\end{multline} Since  $\{\rho_0\le 0\}=\{\rho\le t_1\}$, this implies that
\begin{equation}\label{15.9.17}
\{\rho\le t_1\}\cap \big\{1/2\le \vert z_1\vert\le 1\big\}=B_1\cap \big\{1/2\le \vert z_1\vert\le 1\big\},
\end{equation}
Moreover, by \eqref{6.9.17n}, from \eqref{15.9.17'}  we get
\begin{equation}\label{7.9.17}\begin{split}\Big(\{\rho\le t_1\}\cup B_1\cup \ldots\cup B_{\mu-1}\Big)&\cap \big\{1/2\le \vert z_\mu\vert\le 1\big\}\\&=B_{\mu}\cap \{1/2\le\vert z_\mu\vert\le 1\}\quad\text{for}\quad 2\le \mu\le m.
\end{split}\end{equation}

By  (b), $\im z_\mu=\rho_\mu\big\vert_{U_\mu}$. Therefore $\{\rho_\mu\le 0\}\cap \{1/2\le\vert z_\mu\vert\le 1\}=\{z_\mu\in\Delta_I\}$. Since $B_\mu\cap\{1/2\le\vert z_\mu\vert\le 1\}= \{\rho_\mu\le 0\}\cap \{1/2\le\vert z_\mu\vert\le 1\}$, this means
\begin{equation}\label{7.9.17'}
B_\mu\cap \{1/2\le\vert z_\mu\vert\le 1\}=\{z_\mu\in\Delta_I\}\quad\text{for}\quad 1\le \mu\le m.
\end{equation}

We summarize: By (f), $z_1$ is a $\Cal C^\infty$ diffeomorphism from $U_1$ onto a neighborhood of $\Delta$ and, by definiton of $B_1$, we have $B_1\subseteq \{\vert z_1\vert\le 1\}$. Together with \eqref{1.10.17}, \eqref{15.9.17} and \eqref{7.9.17'} (for $\mu=1$), this shows that $(\{\rho\le t_1\}, B_1)$ is a bump in $X$ (\eqref{23.9.17} is satisfied).
For $2\le \mu\le m$, by (f), $z_\mu$ is a $\Cal C^\infty$ diffeomorphism from $U_\mu$ onto a neighborhood of $\Delta$ and, by definiton of $B_\mu$, we have $B_\mu\subseteq \{\vert z_\mu\vert\le 1\}$. Together with \eqref{1.10.17'}, \eqref{7.9.17} and \eqref{7.9.17'}, this shows that $\big(\{\rho\le t_1\}\cup B_1\cup\ldots\cup B_{\mu-1}, B_\mu\big)$ is a bump in $X$.
Therefore (i') holds. (ii') holds by \eqref{6.9.17n} and \eqref{12.7.17'},  (iii') by (c).
\end{proof}

\begin{thm}\label{5.8.16} Let $X$ be  noncompact connected Riemann surface,  and $\Cal U$ an open covering of $X$. Then there exists a sequence $(B_\mu)_{\mu\in\N}$ of compact subsets of $X$ such that

{\em (a)} for each $\mu\in \N$, $B_\mu$ is contained in at least one set of  $\Cal U$;

{\em (b)} for each $\mu\in \N^*$, $(B_0\cup\ldots\cup B_{\mu-1},B_{\mu})$ is a bump in $X$;

{\em (c)} $X=\bigcup_{\mu\in\N}B_\mu$;

{\em (d)} for each compact  $\Gamma\subseteq X$, there exists $N(\Gamma)\in\N$ with $B_\mu\cap \Gamma=\emptyset$ if $\mu\ge N(\Gamma)$.
\end{thm}

\begin{proof} Since $X$ is a Stein manifold (see, e.g., \cite[Corollary 26.8]{F}), we can find a strictly subharmonic $\mathcal C^\infty$ function $\rho:X\to \R$ such that, for all $\alpha\in\R$, $\{\rho\le \alpha\}$ is compact (see, e.g., \cite[Theorem 5.1.6]{Ho}). Using Morse theory (e.g., \cite[Part I, Corollary 6.8]{M}), we can easily achieve that, moreover, all critical points of $\rho$ are non-degenerate, and, for each $t\in\R$, at most one critical point of $\rho$ lies on $\{\rho=t\}$.
Since $\rho$ is strictly subharmonic (which implies that $\rho$ has no local maxima), then, for each critical value $t$ of $\rho$, there are only two  possibilities: either $\rho$ has precisely one critical point on $\{\rho=t\}$, and this is the point of a strong local minimum of $\rho$,
or $\rho$ has precisely one critical point on $\{\rho=t\}$, and this  is the point of a strong saddle point of $\rho$.

In particular, then there is precisely one point in $X$, $\xi_{\mathrm{min}}$, where $\rho$ assumes its absolute minimum, and this minimum is strong. Therefore, we can find $\varepsilon_0>0$ such that $\{\rho\le \rho(\xi_{\mathrm{min}})+\varepsilon_0\}$ is contained in at least one of the sets of $\Cal U$, and $\rho$ has no critical points in $\{\rho(\xi_{\mathrm{min}})<\rho\le \rho(\xi_{\mathrm{min}})+\varepsilon_0\}$. Set $B_0=\{\rho(\xi_{\mathrm{min}})\le\rho\le \rho(\xi_{\mathrm{min}})+\varepsilon_0\}$.

If $\xi_{\mathrm{min}}$ is the only critical point of $\rho$, then the proof of the theorem can be completed inductively, applying  Lemma \ref{7.8.16} with $\alpha=\rho(\xi_{\mathrm{min}})+\varepsilon_0+N$ and $\beta=\rho(\xi_{\mathrm{min}})+\varepsilon_0+N+1$ for $N=0,1,2, \ldots$ .

If there are further critical points of $\rho$, it remains to complete Lemma \ref{7.8.16} by the  following statement.

\begin{itemize}
\item [(*)] Let $\xi$ be a critical point of $\rho$ with $t:=\rho(\xi)>\rho(\xi_{\mathrm{min}})+\varepsilon_0$. Then there exists $\varepsilon>0$ such that, for each $0<\delta\le \varepsilon$,  we can find an $m$-tuple $(A_1,\ldots,A_m)$ such that $\big(\{\rho\le t-\delta\},A_1,\ldots,A_m\big)$ is a bump extension in $X$, $\{\rho\le t-\delta\}\cup A_1\cup\ldots\cup A_m=\{\rho\le t+\delta\}$, and, for each $1\le\mu\le m$, $A_\mu\cap\{\rho\le t-\delta-1\}=\emptyset$ and  $A_\mu$ is contained in at least one  set of $\Cal U$.
\end{itemize}

\vspace{2mm}
{\em Proof of} (*) {\em if $\xi$ is the point of a  strong local minimum of $\rho$:}

\vspace{2mm}
Then  $\xi$ is an isolated point of $\{\rho\le t\}$. Therefore we can find an open neighborhood $U$ of $\xi$ and an open neighborhood $V$ of $\{\rho\le t\}\setminus\{\xi\}$ such that $U\cap V=\emptyset$ and $U$ is contained in at least one  set of $\mathcal U$. Choose $\varepsilon>0$ such that $\rho$ has no critical points on $V\cap \{t-\varepsilon\le \rho\le t+\varepsilon\}$ and
\[
\{\rho\le t+\varepsilon\}\subseteq U\cup V.
\]

To prove that this $\varepsilon$ has the required property, let $0<\delta\le \varepsilon$ be given.
Then, by Lemma \ref{7.8.16},  we can find
$A_1,\ldots,A_{m-1}$ such that $(\{\rho\le t-\delta\},A_1,\ldots,A_{m-1})$ is a bump extension in $V$, $\{\rho\le t-\delta\}\cup A_1\cup\ldots\cup A_{m-1}=V\cap\{\rho\le t+\delta\}$, and, for $1\le \mu\le m-1$, $A_\mu\cap\{\rho\le t-\delta-1\}=\emptyset$ and  $A_\mu$  is contained in at least one  set of $\Cal U$. Since $U\cap V=\emptyset$, it
 remains to set $A_m=U\cap \{\rho\le t+\delta\}$.

\vspace{2mm}
{\em Proof of} (*) {\em if $\xi$ is the point of a  strong saddle point of $\rho$:}

\vspace{2mm}
By a lemma of Morse (see, e.g., \cite[Part I, Lemma 2.2]{M}), then we can find $R>0$, an open neighborhood $W$ of $\xi$ and a $\Cal C^\infty$ diffeomorphism $w$ from $W$ onto a neighborhood of $\Delta_R:=\big\{\lambda\in\C\,\big\vert\,\vert\lambda\vert\le R\big\}$   such that $w(\xi)=0$ and
\begin{equation}\label{13.8.16'}\rho=t+(\rea w)^2-(\im w)^2\quad\text{on}\quad W.
\end{equation}Moreover, we may choose $W$ so small that
\begin{equation}\label{13.8.16''}
W\cap\{\rho\le t-1\}=\emptyset\text{ and } W\text{ is contained in at least one set of }\Cal U.
\end{equation}
Take $\varepsilon>0$ so small that $\xi$ is the only critical point of $\rho$ on $\{t-\varepsilon\le\rho\le t+\varepsilon\}$. To prove that this $\varepsilon$ has the required property, let $0<\delta\le \varepsilon$ be given.

Choose $0<r\le R/2$ with
\begin{equation}\label{12.8.16}\{\vert w\vert\le r\}\subseteq \{t-\delta<\rho<t+\delta\},
\end{equation}
and take a $\Cal C^\infty$ function $\chi:X\to [0,1]$ with $\chi=1$ on $\{\vert z\vert\le r/2\}$ and $\chi=0$ on $X\setminus\{\vert w\vert\le r\}$. By \eqref{13.8.16'}, then we can find  $c>0$ so small that the functions $\rho_+,\rho_-:X\to\R$ defined by
 \[
 \rho^{}_+=\rho+c\chi\quad\text{and}\quad\rho^{}_-=\rho-c\chi
 \]have the same critical points as $\rho$.
Then $\rho_+^{}=\rho^{}_-=\rho$ on $X\setminus\{\vert w\vert\le r\}$ and $\rho_+^{}\ge \rho\ge\rho_-$  on $X$. Therefore and by \eqref{12.8.16},
\begin{align}\label{12.8.16'}
&\{\rho^{}_+\le t-\delta\}=\{\rho\le t-\delta\},\\
&\label{14.7.17}\{t-\delta\le \rho_+\le t\}\subseteq\{t-\delta\le \rho\le t\}\subseteq\{t-\varepsilon\le \rho\le t+\varepsilon\},\\
&\label{14.7.17'}\{\rho^{}_+\le t-\delta-1\}=\{\rho\le t -\delta-1\},\\
\label{13.8.16}&\{\rho^{}_-\le t+\delta\}=\{\rho\le t+\delta\},\\
&\label{14.7.7''}\{t\le \rho_-^{}\le t+\delta\}\subseteq\{t\le \rho\le t+\delta\}\subseteq\{t-\varepsilon\le \rho\le t+\varepsilon\},\\
&\label{14.7.17'''}\{\rho^{}_-\le t-1\}\supseteq\{\rho\le t-\delta-1\}.
\end{align}
Since $\rho_+^{}(\xi)=\rho(\xi)+c>t $, we have $\xi\not\in \{t-\delta\le \rho_+^{}\le t\}$.
As $\xi$ is the only critical point of $\rho$ in $\{t-\varepsilon\le \rho\le t+\varepsilon\}$ and by \eqref{14.7.17}, this implies that $\rho$ has no critical point in $\{t-\delta\le \rho_+\le t\}$. Since   $\rho_+$ has the same critical points as $\rho$, this means that $\rho_+$ has no critical point in $\{t-\delta\le \rho_+\le t\}$.
Therefore, by Lemma \ref{7.8.16}, by \eqref{12.8.16'} and by \eqref{14.7.17'}, we can find $A_1,\ldots,A_{k}$ such that
\begin{itemize}
\item[(a)] $\big(\{\rho\le t-\delta\},A_1,\ldots,A_{k}\big)$ is a bump extension in $X$,   $\{\rho\le t-\delta\}\cup A_1\cup\ldots\cup A_{k}=\{\rho_+^{}\le t\}$ and, for each  $1\le \mu\le k$, $A_\mu\cap \{\rho\le t-\delta-1\}=\emptyset$ and $A_{\mu}$ is contained in at least one  set of $\mathcal U$.
\end{itemize}
Since $\rho_-^{}(\xi)=\rho(\xi)-c<t $, we have $\xi\not\in \{t\le\rho_-^{}\le t+\delta\}$.
As $\xi$ is the only critical point of $\rho$ in $\{t-\varepsilon\le \rho\le t+\varepsilon\}$ and by \eqref{14.7.7''}, this implies that $\rho$ has no critical points in $\{t\le\rho_-^{}\le t+\delta\}$. Since $\rho_-$ has the same critical points as $\rho$, this means that $\rho_-$ has no critical points in $\{t\le\rho_-^{}\le t+\delta\}$. Therefore, by Lemma \ref{7.8.16}, by \eqref{13.8.16} and by \eqref{14.7.17'''}, we can find $A_{k+2},\ldots,A_m$ such that
\begin{itemize}
\item[(b)] $\big(\{\rho_-\le t\},A_{k+2},\ldots,A_m\big)$ is a bump extension in $X$, $\{\rho_-\le t\}\cup A_{k+2}\cup\ldots\cup A_m=\{\rho\le t+\delta\}$ and, for each  $k+2\le \mu\le m$, $A_\mu\cap \{\rho\le t-\delta-1\}=\emptyset$ and $A_{\mu}$ is contained in at least one  set of $\mathcal U$.
\end{itemize}
Set $z=w/2r$ on $W$. Since $r\le R/2$, then $z$ is a diffeomorphism from $W$ onto a neighborhood of $\Delta$. Since
\begin{equation}\label{16.7.17-} \rho_+=\rho_-=\rho\quad\text{on}\quad X\setminus\big\{\vert z\vert\le 1/2\big\}=X\setminus\big\{\vert w\vert\le r\big\},
\end{equation} it follows from \eqref{13.8.16'} that
\begin{equation}\label{1.10.17''}
\rho_+=\rho_-=\rho=t+4r^2\Big((\rea z)^2-(\im z)^2\Big)\quad\text{on}\quad \big\{1/2\le \vert z\vert\le 1\big\}.
\end{equation}

Set $A_{k+1}=\big\{\rho_-\le t\big\}\cap \big\{\vert z\vert\le 1\big\}$.
Then
\begin{equation}\label{1.10.17-} A_{k+1}\subseteq \big\{\vert z\vert\le 1\big\}.
\end{equation}Since $\rho_-\le \rho_+$, we have
\begin{equation}\label{1.10.17--}
\{\rho_+\le t\}\cap \{\vert z\vert\le 1\}\subseteq A_{k+1}
\end{equation}and
\[\big(\{\rho_+\le t\}\cup A_{k+1}\big)\cap \big\{\vert z\vert\le 1\big\}=\big\{\rho_-\le t\big\}\cap \big\{\vert z\vert\le 1\big\}.
\]Together with \eqref{16.7.17-}, the latter yields
\begin{equation}\label{1.10.17---}\big\{\rho_+\le t\big\}\cup A_{k+1}=\big\{\rho_-\le t\}.
\end{equation}
From \eqref{1.10.17''} it follows that
\begin{equation}\label{1.10.17----}
\big\{\rho_+\le t\big\}\cap \big\{1/2\le \vert z\vert\le 1\big\}=A_{k+1}\cap\big\{1/2\le \vert z\vert\le 1\big\}
=\{z\in\Delta_{II}\}.
\end{equation}

By \eqref{1.10.17-} - \eqref{1.10.17----}, $\big(\{\rho_+\le t\},A_{k+1}\big) $ is a bump in $X$ (condition \eqref{23.9.17'} is satisfied) with $\big\{\rho_+\le t\big\}\cup A_{k+1}=\big\{\rho_-\le t\}$
and such that, by \eqref{13.8.16''}, $A_{k+1}\cap\{\rho\le t-\delta-1\}=\emptyset$ and $A_{k+1}$ is contained in at least one set of $\mathcal U$. Together with (a) and (b) it follows that the $m$-tuple $(A_1,\ldots,A_m)$ has the required properties.
\end{proof}

\section{$Z$-adapted pairs in 1-dimensional complex spaces}\label{22.9.17}

\begin{defn}\label{22.9.17'} Let $X$ be a 1-dimensional complex space, and let $Z$ be a discrete and closed subset of $X$ such that all points of $X\setminus Z$ are smooth. A pair $(\Gamma_1,\Gamma_2)$ will be called a {\bf $Z$-adapted pair in $X$} if $\Gamma_1$ and $\Gamma_2$ are compact subsets of $X$ such that $\Gamma_1\cap \Gamma_2=\Gamma^{\mathrm o }\cup \Gamma$, where
\begin{itemize}
\item $\Gamma^{\mathrm o }\cap \Gamma=\emptyset$;
\item $\Gamma^{\mathrm o }\subseteq Z$;
\item  $\Gamma\cap Z=\emptyset$ and, if $\Gamma\not=\emptyset$, then  $\Gamma$ is the union of  a finite number\footnote{In the applications below, this `finite number' will be one or two.} of pairwise disjoint compact sets  each of which has a basis of contractible open neighborhoods.
\end{itemize}
\end{defn}
\begin{lem}\label{1.9.17} Let $X$ be a 1-dimensional complex space, and let $Z$ be a discrete and closed subset of $X$ such that all points of $X\setminus Z$ are smooth. Let $\pi:\widetilde X\to X$ be the normalization of $X$ (see, e.g., \cite[Ch. VI, \S 4]{L}), and let $(B_1,B_2)$ be a bump in $\widetilde X$ (Def. \ref{3.8.16}). Then there exists a $Z$-adapted pair in $X$,  $(\Gamma_1,\Gamma_2)$,  such that
\begin{align}
&\label{22.9.17'''} \Gamma_1\subseteq \pi(B_1),\\
&\label{30.9.17*} \Gamma_2\subseteq \pi(B_2),\\
&\label{22.9.17''}\Gamma_1\cup \Gamma_2=\pi(B_1\cup B_2).
\end{align}
\end{lem}
\begin{proof} First let $B_1\cap B_2=\emptyset$. Since all points of $X\setminus Z$ are smooth and, hence, $\pi$ is bijective from $\widetilde X\setminus \pi^{-1}(Z)$ onto $X\setminus Z$, then $\pi(B_1)\cap \pi(B_2)\subseteq Z$. Set $\Gamma_1=\pi(B_1)$ and $\Gamma_2=\pi(B_2)$. Then $(\Gamma_1,\Gamma_2)$ is a $Z$-adapted pair  in $X$ (we can take $\Gamma^{\mathrm o}=\pi(B_1)\cap \pi(B_2)$ and $\Gamma=\emptyset$ in Def. \ref{22.9.17'}) with \eqref{22.9.17'''}-\eqref{22.9.17''}.

Now let $B_1\cap B_2\not=\emptyset$. Then (by Definition \ref{3.8.16}) we have an open neighborhood $U$ of $B_2$ and a diffeomorphic map, $z$, from $U$ onto an open neighborhood of $\Delta$ satisfying \eqref{23.9.17'''}, \eqref{30.9.17a}  and one of the relations \eqref{23.9.17} or \eqref{23.9.17'}. Since $\pi^{-1}(Z)$ is discrete and closed in $\widetilde X$, we can find $1/2<r<R<1$ such that
\begin{equation}\label{23.9.17--}
\pi^{-1}(Z)\cap \{r\le \vert z\vert\le R\}=\emptyset\quad\text{and, hence,}\quad Z\cap \pi\big(\{r\le \vert z\vert\le R\}\big)=\emptyset.
\end{equation}
Define compact subsets $K,K_1, K_2$ of $\widetilde X$ and compact subsets $\Gamma, \Gamma_1,\Gamma_2,\Gamma^{\mathrm o}$ of $X$, by
\begin{align*}&K=B_1\cap B_2\cap\big\{r\le \vert z\vert\le R\big\},\;
K_1=B_1\setminus\{\vert z\vert<R\},\;
K_2=B_2\cap \{\vert z\vert\le r\},\\
& \Gamma=\pi(K),\; \Gamma_1=\pi(K\cup K_1),\;\Gamma_2=\pi(K\cup K_2),\;\Gamma^{\mathrm o}=\pi(K_1)\cap\pi( K_2).
\end{align*}
Then
\begin{multline}\label{25.9.17''}
\Gamma_1\cap \Gamma_2=\big(\pi(K)\cup\pi(K_1)\big)\cap\big(\pi(K)\cup\pi(K_2)\big)\\=\pi(K)\cup\big(\pi(K_1)\cap\pi(K)\big)\cup\big(\pi(K_1)\cap\pi(K_2)\big)
=\pi(K)\cup\big(\pi(K_1)\cap\pi(K_2)\big) =\Gamma\cup\Gamma^{\mathrm 0}.
\end{multline}
Since all points in $X\setminus Z$ are smooth points of $X$, $\pi$ is bijective  from $\widetilde X\setminus\pi^{-1}(Z)$ onto $X\setminus Z$. As $K_1\cap K_2=\emptyset$, this implies that
\begin{equation}\label{25.9.17}
\Gamma^{\mathrm o}\subseteq Z.
\end{equation} Since $\Gamma\subseteq \pi\big(\{r\le \vert z\vert\le R\}\big)$, from \eqref{23.9.17--} we get $\Gamma\cap Z=\emptyset$ and, hence, by \eqref{25.9.17},
\begin{equation}\label{25.9.17'}
\Gamma\cap \Gamma^{\mathrm o}=\emptyset.
\end{equation}
From \eqref{23.9.17} resp. \eqref{23.9.17'} it follows that
\[
K=\begin{cases}\{z\in\Delta_I\}\cap \{r\le \vert z\vert\le R\}\quad&\text{in case  \eqref{23.9.17}},\\
\{z\in\Delta_{II}\}\cap \{r\le \vert z\vert\le R\}\quad&\text{in case  \eqref{23.9.17'}}.
\end{cases}
\]
Since $z$ diffeomorphic, this implies  that, in case \eqref{23.9.17}, $K$ has a basis of contractible open neighborhoods, and, in case \eqref{23.9.17'}, $K$ is the union of two disjoint compact sets each of which has a basis of contractible open neighborhoods.
 Since,  by \eqref{23.9.17--}, $\pi$ is homeomorphic from an open neighborhood of $K$ onto an open neighborhood of $\Gamma$, the same is true for $\Gamma$.
Together with \eqref{25.9.17''}-\eqref{25.9.17'}, this shows that $(\Gamma_1,\Gamma_2)$ is a $Z$-adapted pair  in $X$.

Moreover, since $K\cup K_1\subseteq B_1$ and $K\cup K_2\subseteq  B_2$, relations \eqref{22.9.17'''}, \eqref{30.9.17*} and ``$\subseteq$'' in \eqref{22.9.17''} are clear. It remains to prove ``$\supseteq$'' in \eqref{22.9.17''}. For that
let $\zeta\in \pi(B_1\cup B_2)$. Then $\zeta=\pi(\widetilde\zeta)$ for some $\widetilde\zeta\in B_1\cup B_2$.

If $\widetilde\zeta\in B_1$, then at least one of the following holds:
\[\text{(a) } \widetilde\zeta\in K_1,\quad\text{(b) } \widetilde\zeta\in B_1\cap\{\vert z\vert\le r\}, \quad\text{(c) }\widetilde \zeta\in B_1\cap \{r\le \vert z\vert\le R\}.\]
In case (a), $\zeta\in \Gamma_1$. In case (b), it follows from \eqref{30.9.17a} that $\widetilde\zeta\in B_2\cap\{\vert z\vert\le r\}=K_2$ and, hence, $\zeta\in \Gamma_2$. In case (c), it follows from \eqref{23.9.17} or \eqref{23.9.17'} that $\widetilde \zeta\in B_1\cap B_2\cap\big\{r\le \vert z\vert\le R\big\}=K$ and, therefore, $\zeta\in \Gamma\subseteq \Gamma_1\cap \Gamma_1$.

If $\widetilde\zeta\in B_2$,  then at least one of the following holds:
\[\text{(a) } \widetilde\zeta\in K_2,\quad\text{(b) } \widetilde\zeta\in B_2\cap\{r\le\vert z\vert\le R\}, \quad\text{(c) }\widetilde \zeta\in B_2\cap \{R\le \vert z\vert\le 1\}.\]
In case (a), $\zeta\in \Gamma_2$. In case (b), by \eqref{23.9.17} or \eqref{23.9.17'}, $\widetilde\zeta\in B_1\cap B_2\cap\{r\le \vert z\vert\le R\}=K$ and, hence, $\zeta\in \Gamma_1\cap \Gamma_2$. In case (c), again by \eqref{23.9.17} or \eqref{23.9.17'},
$\widetilde \zeta\in B_1\cap\big\{R\le \vert z\vert\le 1\big\}\subseteq K_1$ and, therefore, $\zeta\in \Gamma_1$.
\end{proof}

\section{Jordan stable points}\label{8.6.16}

In this section, $X$ is a complex space (of arbitrary dimension),  and $A:X\to \mathrm{Mat}(n\times n,\C)$ is a holomorphic map.

\begin{defn}\label{8.6.16n'}
A point $\xi\in X$ will be called {\bf Jordan stable} for $A$ if there exists a neighborhood $U$ of $\xi$ such that the following two conditions are satisfied:

(a) there are holomorphic functions $\lambda_1,\ldots,\lambda_m:U\to \C$ such that, for each $\zeta\in U$, $\lambda_1(\zeta),\ldots,\lambda_m(\zeta)$ are the different eigenvalues of $A(\zeta)$;

(b) there is a holomorphic map $T:U\to\mathrm{GL}(n\C)$ such that, for all $\zeta\in U$, $T(\zeta)^{-1}A(\zeta)T(\zeta)$ is in Jordan normal form.\footnote{Equivalently, one could define: $\xi$ is Jordan stable for $A$ if there exists a neighborhood $U$ of $\xi$ such that the number of different eigenvalues of $A(\zeta)$ is the same for all $\zeta\in U$ and,  for all integers $1\le k\le n$, the number of Jordan blocks in the Jordan normal forms of $A(\zeta)$ is the same for all $\zeta\in U$. This was proved by G. P. A. Thijsse \cite{T} (see also \cite[Lemma 5.3]{Le1}).}
\end{defn}

\begin{prop}\label{11.6.16}
The points in $X$ which are \underline{not} Jordan stable for $A$ form a nowhere dense analytic   subset of $X$. (If $X$ is 1-dimensional, this means that this set is discrete and closed in $X$.)
\end{prop}

This proposition can be found in \cite[Theorem 5.5]{Le1}. If $X$ is 1-dimensional and smooth, it was first proved by
H. Baumg\"artel \cite{B1}, \cite[Kap. 5, \S 7]{B2}, \cite[5.7]{B4}. If $X$ is of arbitrary dimension and smooth, H. Baumg\"artel \cite{B3}, \cite[S 3.4]{B4} obtained the slightly weaker statement that the points in $X$ which are not Jordan stable for $A$ are \underline{contained} in a nowhere dense analytic   subset of $X$.

\begin{thm}\label{26.1.16'}  Let  $\xi\in X$ be Jordan stable for $A$. Then there exist a neighborhood $U$ of $\xi$ and a holomorphic map $H:U\to GL(n,\C)$ such that (cp. Def. \ref{26.9.17n})
\begin{equation}\label{27.1.16--}
H(\zeta)^{-1}\big(\mathrm{Com\,}A(\zeta)\big)H(\zeta)=\mathrm{Com\,}A(\xi)\quad\text{for all}\quad \zeta\in U.
\end{equation}
\end{thm}
\begin{proof}
Let $U$, $\lambda_j$, $1\le j\le m$, and $T$ be as in Definition \ref{8.6.16n'}. Set $J(\zeta)=T(\zeta)^{-1}A(\zeta)T(\zeta)$ and $H(\zeta)=T(\zeta)T(\xi)^{-1}$ for $\zeta\in U$. We may assume that $U$ is conneceted. Since $J$ is continuous and, for each $\zeta\in U$, $\lambda_1(\zeta), \ldots, \lambda_m(\zeta)$ are the different eigenvalues of $J(\zeta)$, and since $J(\zeta)$ is in Jordan normal form, then, after a possible change of the numbering, there are integers $n_1,\ldots,n_m\ge 1$ and matrices $M_j\in \mathrm{Mat}(n_j\times n_j,\C)$, $1\le j\le m$, in Jordan normal form and with the only eigenvalue $0$, such that, for each $\zeta\in U$, $J(\zeta)$ is the block diagonal matrix with the diagonal $\lambda_1(\zeta)I_{n_1}+M_1,\ldots,\lambda_m(\zeta) I_{n_m}+M_m$.

Since the eigenvalues $\lambda_1(\zeta),\ldots,\lambda_m(\zeta)$ are pairwise different, then \cite[Ch. VIII, \S 1]{Ga},  a matrix $\Theta\in\mathrm{Mat}(n\times n,\C)$ belongs to $\mathrm{Com\,} J(\zeta)$ if and only if $\Theta$ is a block diagonal matrix with matrices $Z_j\in \mathrm{Com\,}\big(\lambda_j(\zeta)I_{n_j}+M_{j}\big)$, $1\le j\le m$,
on the diagonal.  Obviously, $\mathrm{Com\,}\big(\lambda_j(\zeta)I_{k_j}+M_{j}\big)=\mathrm{Com\,}\big(\lambda_j(\xi)I_{k_j}+M_{j}\big)$ for all $\zeta\in U$. Therefore,  $\mathrm{Com\,}J(\zeta)=\mathrm{Com\,}J(\xi)$ for all $\zeta\in U$.
 \eqref{27.1.16--} now follows by \eqref{26.8.16+}.
\end{proof}

\begin{rem}\label{27.9.17*} Let $Z$ be the set of all points in $X$ which are not Jordan stable for $A$ (which is, by Proposition \ref{11.6.16}, a nowhere dense analytic subset of $X$). Assume that $X\setminus Z$ is connected, fix some point $\xi\in X$, and denote  by $N$ the normalizer of $\mathrm{GCom\,}A(\xi)$ in $\mathrm{GL}(n,\C)$, i.e., the complex Lie group of all $\Phi\in\mathrm{GL}(n,\C)$ with
$\Phi^{-1}\big(\mathrm{GCom\,}A(\xi)\big)\Phi=\mathrm{GCom\,}A(\xi)$. Then Theorem \ref{26.1.16'} implies that the family
\[
\big\{\mathrm{GCom\,}A(\zeta)\big\}_{\zeta \in X\setminus Z}
\]
is a holomorphic $N$-principal bundle of complex Lie groups, with the characteristic fiber $\mathrm{GCom\,}A(\xi)$. Note that  $\mathrm{GCom\,}A(\xi)$ is connected (as easy to see -- cp. \cite[Lemma 4.2]{Le2}), whereas $N$ need not be connected (cp. Remark \ref{17.3.16'} below).
\end{rem}

Theorem \ref{26.1.16'} immediately implies
\begin{cor}\label{26.7.17m} For each open set $W\subseteq X$ which contains only Jordan stable points of $A$, we have $\widehat{\mathcal O}^{\mathrm{Com\,}A}(W)={\mathcal C}^{\mathrm{Com\,}A}(W)$.
\end{cor}

\begin{thm}\label{23.7.16}Let $W\subseteq X$ be an open set which contains only Jordan stable points of $A$. Further suppose that $W$ is contractible. Fix some point $\xi\in W$. Then there exists a continuous map
$T:W\to \mathrm{GL}(n,\C)$  such that
\begin{equation}\label{27.1.16--neu}
T^{-1}(\zeta)\big(\mathrm{Com\,}A(\zeta)\big)T(\zeta)=\mathrm{Com\,}A(\xi)\quad\text{for all}\quad \zeta\in W.
\end{equation} If, moreover, $W$ is Stein, then this $T$ can be chosen holomorphically on $W$.
\end{thm}
\begin{proof} Since $W$ is connected, by Theorem \ref{26.1.16'}, for each $\eta\in W$, we can find a neighborhood $U_\eta\subseteq W$ of $\eta$
and a holomorphic map $H_\eta:U_\eta\to GL(n,\C)$ such that
\begin{equation}\label{27.8.16}H_\eta^{-1}(\zeta)\big(\mathrm{Com\,}A(\zeta)\big)H_\eta^{}(\zeta)=\mathrm{Com\,}A(\xi)\quad\text{for all}\quad\zeta\in U_\eta.
\end{equation}
Let $N$ the normalizer of $\mathrm{GCom\,}A(\xi)$ in $\mathrm{GL}(n,\C)$, i.e.,  the complex Lie group of all $\Phi\in\mathrm{GL}(n,\C)$ with
$\Phi^{-1}\big(\mathrm{GCom\,}A(\xi)\big)\Phi=\mathrm{GCom\,}A(\xi)$.
Then, by \eqref{27.8.16},
\begin{equation*}
H^{-1}_\eta(\zeta)H^{}_\tau(\zeta)\in N\quad\text{for all}\quad \zeta\in U_\eta\cap U_\tau\quad\text{and all}\quad\eta,\tau\in W,
\end{equation*} i.e., the family $\{H^{-1}_\eta H^{}_\tau\}_{\eta\tau}$ is an $\Cal O^N$-cocycle with the covering $\{U_\eta\}_{\eta\in W}$. Since $W$ is contractible, this cocycle splits as a $\Cal C^N$-cocycle \cite[Corollary 11.6]{St},
i.e., there is a family $T_\eta\in \Cal C^N(U_\eta)$ such that $T^{-1}_\eta T^{}_\tau=H_\eta^{}H_\tau^{-1}$ on $U_\eta\cap U_\tau$. Then
\[
H_\eta^{}T^{}_\eta=H_\tau^{}T^{}_\tau\quad\text{on}\quad U_\eta\cap U_\tau,\quad \eta,\tau\in W.
\]Hence, there is a well-defined continuous map $T:X\to \mathrm{GL}(n,\C)$ with
\[
T=H_\eta^{}T^{}_\eta\quad\text{on}\quad U_\eta,\quad \eta\in W.
\]By \eqref{27.8.16} and since $T_\eta(\zeta)\in N$, this $T$ satisfies \eqref{27.1.16--neu}.
If $W$ is Stein, by Grauert's Oka principle \cite[Satz 6]{Gr} (see also \cite[Theorem 8.2.1 (i)]{Fc}), the maps $T_\eta$ can be chosen holomorphically, which implies that  $T$ is holomorphic.
\end{proof}
\begin{cor}\label{9.8.16+}Under the hypotheses of Theorem \ref{23.7.16},   $\widehat{\Cal O}^{\mathrm{GCom\,} A}(W)$, endowed with   the topology of uniform convergence on the compact subsets of $W$, is connected.
\end{cor}
\begin{proof}  Let  $G:=\mathrm{GCom\,} A(\xi)$, and let $\Cal C^{G}(W)$ be also endowed with the topology of uniform convergence on the compact subsets of $W$.
 Since $W$ is contractible, each element of $\Cal C^{G}(W)$ can be connected by a continuous path in $\Cal C^{G}(W)$ with a constant map. It is easy to see that $G$ is connected \cite[Lemma 4.2]{Le2}. Therefore, this implies  that $\Cal C^{G}(W)$ is connected. Since, by \eqref{27.1.16--},  $\Cal C^{\mathrm{GCom\,} A}(W)$ and $\Cal C^{G}(W)$ are isomorphic as topological groups, it follows that $\Cal C^{\mathrm{GCom\,} A}(W)$ is connected. Since $\widehat{\Cal O}^{\mathrm{GCom\,} A}(W)={\Cal C}^{\mathrm{GCom\,} A}(W)$ (Corollary \ref{26.7.17m}), this completes the proof.
\end{proof}

 \begin{rem}\label{17.3.16'} Assume that $X$ is a domain in the complex plane and all points of $X$ are Jordan stable for $A$. Let $\xi\in X$. If, moreover, $X$ is simply connected (and, hence, contractible), then, by Theorem \ref{23.7.16}, we can find a holomorphic map
$T:X\to \mathrm{GL}(n,\C)$  such that $T^{-1}(\zeta)\big(\mathrm{Com\,}A(\zeta)\big)T(\zeta)=\mathrm{Com\,}A(\xi)$ for all $\zeta\in X$.

Question: Is this true also if $X$ is not  simply connected?

If the normalizer, $N$, of $\mathrm{GCom\,}A(\xi)$ in $\mathrm{GL}(n,\C)$, which appears in the proof of Theorem \ref{23.7.16}, is connected, this is the case. This follows by the same proof, using \cite[Satz 7]{Gr} (see also, \cite[Theorem 8.2.1 (iii)]{Fc}) (saying that $H^1(Y,\Cal O^{G})=0$ for each connected complex Lie group $G$ and each noncompact connected Riemann surface $Y$).
However, this is not always the case.

For example, assume that $A(\xi)=\begin{pmatrix}1&0\\0&0\end{pmatrix}$. Then
\begin{equation}\label{12.6.16n'}
N=\bigg\{\begin{pmatrix}a&0\\0&d\end{pmatrix}\;\bigg\vert\;a,d\in\C\setminus\{0\}\bigg\}\cup \bigg\{\begin{pmatrix}0&b\\c&0\end{pmatrix}\;\bigg\vert\;b,c\in\C\setminus\{0\}\bigg\},
\end{equation} which
is not connected.

{\em Proof of} \eqref{12.6.16n'}. It is easy to see  that
\begin{equation}\label{12.6.16n+}
\mathrm{GCom\,}\begin{pmatrix}1&0\\0&0\end{pmatrix}=\bigg\{\begin{pmatrix}a&0\\0&d\end{pmatrix}\;\bigg\vert\;a,d\in \C\setminus\{0\}\bigg\}.
\end{equation}
To
prove ``$\supseteq $'' in \eqref{12.6.16n'},  we therefore only have to prove that, for all $b,c\in\C\setminus\{0\}$,
\begin{equation}\label{12.6.16n-}\begin{pmatrix}0&b\\c&0\end{pmatrix}\in N.
\end{equation}Let $b,c\in \C\setminus\{0\}$ be given. Then, for all $\alpha,\delta\in \C\setminus\{0\}$,
\[
\begin{pmatrix}0&b\\c&0\end{pmatrix}^{-1}\begin{pmatrix}\alpha&0\\0&\delta\end{pmatrix}
\begin{pmatrix}0&b\\c&0\end{pmatrix}=
\begin{pmatrix}0&c^{-1}\\b^{-1}&0\end{pmatrix}\begin{pmatrix}0&\alpha b\\\delta c&0\end{pmatrix}=\begin{pmatrix}\delta&0\\0&\alpha\end{pmatrix}.
\]
In view of \eqref{12.6.16n+}, this  implies that
\[\begin{pmatrix}0&b\\c&0\end{pmatrix}^{-1}\mathrm{GCom}\begin{pmatrix}1&0\\0&0\end{pmatrix}
\begin{pmatrix}0&b\\c&0\end{pmatrix}= \mathrm{GCom}\begin{pmatrix}1&0\\0&0\end{pmatrix},
\]
 which means \eqref{12.6.16n-}, by definition of $N$.

To prove ``$\subseteq $'' in \eqref{12.6.16n'}, let
$\begin{pmatrix}a&b\\c&d\end{pmatrix}\in N$ be given. We have to prove that
\begin{equation}\label{27.9.17'}
\text{either}\qquad b=c=0,\qquad\text{or}\qquad a=d=0.
\end{equation}
By definition of $N$ and by \eqref{12.6.16n+}, we can find $\alpha, \delta\in \C\setminus \{0\}$ with
\[\begin{pmatrix}a&b\\c&d\end{pmatrix}^{-1}\begin{pmatrix}2&0\\0&1\end{pmatrix}\begin{pmatrix}a&b\\c&d\end{pmatrix}
=\begin{pmatrix}\alpha&0\\0&\delta\end{pmatrix},\]  which implies
\begin{equation}\label{27.9.17}\begin{pmatrix}a&b\\c&d\end{pmatrix}^{-1}\begin{pmatrix}1&0\\0&0\end{pmatrix}\begin{pmatrix}a&b\\c&d\end{pmatrix}
=\begin{pmatrix}1+\alpha&0\\0&1+\delta\end{pmatrix}
\end{equation}and, hence,
\begin{equation}\label{12.6.16n+++}\begin{pmatrix}a&b\\0&0\end{pmatrix}=\begin{pmatrix}a(1+\alpha) &b(1+\delta)\\c(1+\alpha) &d(1+\delta) \end{pmatrix}.
\end{equation}
By \eqref{27.9.17}, $\begin{pmatrix}1+\alpha &0\\0&1+\delta\end{pmatrix}$ is of rank $1$. Therefore
\[ \text{either}\quad 1+\alpha\not=0\quad\text{and}\quad 1+\delta=0,\qquad\text{or}\qquad 1+\alpha=0\quad\text{and}\quad 1+\delta\not=0.
\] By \eqref{12.6.16n+++}, this proves \eqref{27.9.17'}. \qed
\end{rem}

\section{Proof of Theorem \ref{18.10.16} }\label{26.7.17}

In this section, we prove Theorem \ref{18.10.16}. From now on, $X$ is a 1-dimensional Stein space, and $A:X\to \mathrm{Mat\,}(n\times n,\C)$ is a holomorphic map. 

By Corollary \ref{4.10.17*}, it is sufficient to prove the following

\begin{thm}\label{27.9.17''} $H^1(X,\widehat{\mathcal O}^{\mathrm{GCom\,}A})=0$.
\end{thm}

\begin{defn}\label{20.8.17} Let $Z$ be a  discrete and closed subset of $X$.

If $U$ is a nonempty open subset of $X$, then we define:

$\widehat{\mathcal O}^{\mathrm{Com\,} A}(U,Z)$ is the  subalgebra  of all $f\in \widehat{\mathcal O}^{\mathrm{Com\,} A}(U)$ such that there exists an open neighborhood $U_Z$ of $Z$  (depending on $f$) with $f=0$ on $U_Z\cap U$.

$\widehat{\mathcal O}^{\mathrm{GCom\,}A}(U,Z)$ is the subgroup of all  $f\in \widehat{\mathcal O}^{\mathrm{GCom\,} A}(U)$ such that there exists an open neighborhood $U_Z$  of $Z$  (depending on $f$) with $f=I$ on $U_Z\cap U$.

The so defined sheaves on $X$ will be denoted by  $\widehat{\mathcal O}^{\mathrm{Com\,}A}(\cdot,Z)$ and $\widehat{\mathcal O}^{\mathrm{GCom\,}A}(\cdot,Z)$, respectively.
\end{defn}

\begin{defn}\label{27.9.17'''} Let $Z$ be a  discrete and closed subset of $X$. An open covering of $X$, $\Cal  U=\{U_{ij}\}_{i,j\in I}$, will be called {\bf $Z$-adapted} if
\begin{equation}\label{27.9.17-}
U_i\cap U_j\cap Z=\emptyset\quad\text{for all }i,j\in I\text{ with }i\not=j.
\end{equation}
\end{defn}
\begin{lem}\label{27.9.17--}Let $Z$ be a  discrete and closed subset of $X$. Then each open covering of $X$ admits a $Z$-adapted refinement.
\end{lem}
\begin{proof}
Let an open covering $\Cal  U=\{U_{i}\}_{i\in I}$ of $X$ be given.  Since $Z$ is discrete and closed in $X$, we can find a family $\{U^*_\xi\}_{\xi\in Z}$ such that, for each $\xi\in Z$, $U^*_\xi$ is an open neighborhood of $\xi$, and
$U^*_\xi\cap U^*_\eta=\emptyset$ for all $\xi,\eta\in Z$ with $\xi\not=\eta$. Further, we set $U^*_i=U_i\cap (X\setminus Z)$ for $i\in I$. Then
$\{U^*_\alpha\}_{\alpha\in I\cup Z}$ (we may assume that $I\cap Z=\emptyset$) is a $Z$-adapted refinement of $\Cal U$.
\end{proof}

\begin{rem}\label{27.9.17+} To prove Theroem \ref{27.9.17''}, it is sufficient to find a discrete and closed subset $Z$ of $X$ such that
\begin{equation}\label{27.9.17++}
H^1\big(X,\widehat{\mathcal O}^{\mathrm{GCom\,}A}(\cdot,Z)\big)=0.
\end{equation}
Indeed, assume $Z$ is such a set and $f$ is an $\widehat{\mathcal O}^{\mathrm{GCom\,}A}$-cocycle.  Let $\mathcal U$ be the covering of $f$. Then, by Lemma \ref{27.9.17--}, we can find a $Z$-adapted open covering of $X$, $\mathcal U^*$, which is a refinement of $\mathcal U$. Let $f^*$ be a  $(\mathcal U^*,\widehat{\mathcal O}^{\mathrm{GCom\,}A})$-cocycle, which is induced by $f$. Since $\mathcal U^*$ is $Z$-adapted, then $f^*$ can be interpreted as a  $\big(\mathcal U^*,\widehat{\mathcal O}^{\mathrm{GCom\,}A}(\cdot,Z)\big)$-cocycle, and it follows from \eqref{27.9.17++} that $f^*$ splits as an $\widehat{\mathcal O}^{\mathrm{GCom\,}A}(\cdot,Z)$-cocycle. As $\widehat{\mathcal O}^{\mathrm{GCom\,}A}(\cdot,Z)$ is a subsheaf of $\widehat{\mathcal O}^{\mathrm{GCom\,}A}$, this means in particular that $f^*$ splits as an  $\widehat{\mathcal O}^{\mathrm{GCom\,}A}$-cocycle.
\end{rem}

\begin{lem}\label{27.8.17'} Let $Z$ be a discrete and closed subset of $X$, which contains all nonsmooth points of $X$ and all points of $X$ which are not Jordan stable for $A$ (by Proposition \ref{11.6.16} such $Z$ exist), and let $(\Gamma_1,\Gamma_2)$ be a $Z$-adapted pair  in $X$ (Def. \ref{22.9.17'}). Then:

{\em (i)} Let $U$ be an open neighborhood of $\Gamma_1\cap \Gamma_2$ and $f\in  \widehat{\mathcal O}^{\mathrm{GCom\,}A}(U,Z)$. Then there exist  a section $\widetilde f\in \widehat{\mathcal O}^{\mathrm{GCom\,}A}(X,Z)$ and an open neighborhood $V\subseteq U$  of $\Gamma_1\cap \Gamma_2$ such that $\widetilde f=f$ on $V$ and $\widetilde f=I$ on $X\setminus U$.

{\em (ii)} Let $U_1$ be an open neighborhood of $\Gamma_1$,   $f_1\in\widehat{\mathcal O}^{\mathrm{GCom\,}A}(U_1,Z)$, and let $V$ be an open neighborhood of $\Gamma_1\cap \Gamma_2$. Then there exist an open neighborhood $W$ of $\Gamma_1\cup \Gamma_2$, a section $f\in \widehat{\mathcal O}^{\mathrm{GCom\,}A}(W,Z)$ and an open neighborhood $U_1'$ of $\Gamma_1$ such that $U'_1\subseteq U_1$ and $f=f_1$ on $(U'_1\cap W)\setminus V$.

{\em (iii)}  Let $\Cal U=\{U_i\}_{i\in I}$ be an open covering of $X$ such that $\Gamma_2$ is contained in at least one set of $\Cal U$. Let $f\in Z^1\big(\Cal U,\widehat{\mathcal O}^{\mathrm{GCom\,}A}(\cdot,Z)\big)$ such that, for some open neighborhood $W_1$ of $\Gamma_1$, $f\big\vert_{W_1}$ splits. Then there exist an open neighborhood $W$  of $\Gamma_1\cup \Gamma_2$ such that also $f\big\vert_{W}$ splits.
\end{lem}
\begin{proof} (i) Let $\Gamma^{\mathrm o}$ and $\Gamma$ be as in Definition \ref{22.9.17'}.
If $\Gamma=\emptyset$, i.e., $\Gamma_1\cap \Gamma_2\subseteq Z$, then, by definition of $\widehat{\mathcal O}^{\mathrm{GCom\,}A}(U,Z)$, there is a neighborhood $V\subseteq U$ of $\Gamma_1\cap \Gamma_2$ such that $f=I$ on $V$, and we can define $\widetilde f\equiv I$.

Therefore, we may assume that $\Gamma\not=\emptyset$. Then,  shrinking $U$, we can achieve that
 $U=W^{\mathrm o}\cup W$, where $W^{\mathrm o}$ and $W$ are open neighborhoods of $\Gamma^{\mathrm o}$ and $\Gamma$, respectively, such that $W^{\mathrm o}\cap W=\emptyset$, $f=I$ on $W^{\mathrm o}$,   $W$  consists of a finite number of connected components each of which is contractible, and $W\cap Z=\emptyset$. Then all points of $W$ are   Jordan stable for $A$. Hence, by Corollary
  \ref{9.8.16+}, $\widehat{\mathcal O}^{\mathrm{GCom\,}A}(W)$ is connected. Therefore, we can find a continuous map $\theta:[0,1]\times W\to \mathrm{GL}(n,\C)$ such that
\begin{align*}&\theta(t,\cdot)\in \widehat{\mathcal O}^{\mathrm{GCom\,}A}(W)\quad\text{for all}\quad 0\le t\le 1,\\
&\theta(0,\cdot)=f\big\vert_W\quad\text{and}\quad \theta(1,\cdot)= I\quad\text{on}\quad W.
\end{align*}Choose open neighborhoods  $W''\subseteq W'\subseteq W$  of $\Gamma$ such that $W''$ is relatively compact in $W'$, and $W'$ is relatively compact in $W$. Then we can find a partition of $[0,1]$,
$0=t_1<t_2<\ldots<t_m=1$, so fine  that the maps $g_j\in \widehat{\Cal O}^{\mathrm{Com\,}A}(W')$, $1\le j\le m-1$, defined by $g_j(\zeta):=\theta(t_j,\zeta)\theta(t_{j+1},\zeta)^{-1}-I$ satisfy
$\Vert g_j\Vert\le 1/2$ on $W'$. Then $I+g_j\in \widehat{\Cal O}^{\mathrm{GCom\,}A}(W')$ and
\begin{equation}\label{30.8.17}
f=\big(I+g_1\big)\cdot\ldots\cdot\big(I+g_{m-1}\big)\quad\text{on}\quad W'.
\end{equation}
Choose a  continuous function $\chi:X\to [0,1]$ such that $\chi=1$ in $W''$ and $\chi=0$ on a neighborhood of  $X\setminus W'$, and define $\widetilde f\in \widehat{\mathcal O}^{\mathrm{GCom\,}A}(X)$ by
\[
\widetilde f^{}(\zeta)=\begin{cases}\big(I+\chi(\zeta)g_1(\zeta)\big)\cdot\ldots\cdot\big(I+\chi(\zeta)g_{m-1}(\zeta)\big)\quad&\text{if}\quad \zeta\in W',\\
I&\text{if}\quad \zeta\in X\setminus W'.
\end{cases}
\]
Since $X\setminus U\subseteq X\setminus W'$, then $\widetilde f=I$ on $X\setminus U$.
Since $\chi=1$ on $W''$, it follows from \eqref{30.8.17} that $\widetilde f=f$ on $W''$.
Moreover, as $W^{\mathrm o}\subseteq X\setminus W'$, we have $\widetilde f=I=f$ on $W^{\mathrm o}$. Hence,
$\widetilde f=f$ on the neighborhood $V:=W^{\mathrm o}\cup W''$ of $\Gamma_1\cap \Gamma_2$.

(ii) Since $U_1\cap V$ is a neighborhood of $\Gamma_1\cap \Gamma_2$, by part (i) of the lemma, we can find  a section $\widetilde f\in \widehat{\mathcal O}^{\mathrm{GCom\,}A}(X,Z)$ and an open neighborhood $V'\subseteq U_1\cap V$  of $\Gamma_1\cap \Gamma_2$ such that
$\widetilde f=f_1$ on $V'$ and $\widetilde f=I$ on $X\setminus (U_1\cap V)$. Choose open neighborhoods $U_1'$ of $\Gamma_1$ and  $U_2^{}$ of $\Gamma_2$ such that $U'_1\cap U^{}_2\subseteq V'$ and  $U'_1\subseteq U_1$. Further choose an open neighborhood $U'_2$ of $\Gamma_2$ such that $\overline{U'_2}\subseteq U^{}_2$. Then $W:=U'_1\cup U'_2$ is an open neighborhood of $\Gamma_1\cup \Gamma_2$. Consider the open sets
\[W_1:=W\setminus\overline{U'_2}\quad\text{and}\quad W_2:=W\cap U_2.
\] Then $W=W_1\cup W_2$. Moreover, $W^{}_1\subseteq U_1'$, $W_2\subseteq U_2$ and therefore $W_1\cap W_2\subseteq U_1'\cap U_2\subseteq V'_{}$, which implies that  $f^{}_1 \widetilde f^{-1}_{}=I$ on $W_1\cap W_2$. Hence, there is a well-defined section $f\in\widehat{\mathcal O}^{\mathrm{GCom\,}A}(W,Z)$ such that $f=f_1^{}\widetilde f_{}^{-1}$ on $W_1$  and
$f=I$ on $W_2$. It remains to prove that
\begin{equation}\label{31.1.18'}
f=f_1\quad\text{on}\quad (U'_1\cap W)\setminus V.
\end{equation}First observe that $(U'_1\cap W)\setminus V\subseteq X\setminus (U_1\cap V)$ and, hence,
\begin{equation}\label{31.1.18}
\widetilde f=I\quad\text{on}\quad (U'_1\cap W)\setminus V.
\end{equation} Moreover, $U'_1\cap W=(U'_1\cap W_1)\cup(U'_1\cap W\cap \overline{U'_2})\subseteq (U'_1\cap W_1)\cup V\subseteq W_1\setminus V$, which implies that
$(U'_1\cap W)\setminus V\subseteq W_1$ and, hence, $f=f_1^{}\widetilde f_{}^{-1}$ on $(U'_1\cap W)\setminus V$. Together with \eqref{31.1.18}, this proves \eqref{31.1.18'}.

(iii)  Since $f\big\vert_{W_1}$ splits, we have a family $f_i\in\widehat{\mathcal O}^{\mathrm{GCom\,}A}(U_i\cap W_1,Z)$ with
\begin{equation}\label{31.8.17'}
f_{ij}^{}=f_i^{}f_j^{-1}\quad \text{on}\quad U_i\cap U_j\cap W_1.
\end{equation}By hypothesis, $\Gamma_2\subseteq U_{i_0}$ for some $i_0\in I$. Then,  by \eqref{31.8.17'},
\begin{equation*}
f_i^{-1}f^{}_{ii_0}=f_i^{-1}f^{}_{ij}f^{}_{ji_0}=f_j^{-1} f^{}_{j i_0}\quad \text{on}\quad U_i\cap U_j\cap W_1\cap U_{i_0}.
\end{equation*}Therefore, we have  a well-defined section $g\in\widehat{\mathcal O}^{\mathrm{GCom\,}A}(W_1\cap U_{i_0},Z)$ with
\begin{equation}\label{19.9.17'}
g=f_i^{-1}f^{}_{ii_0}\quad\text{on}\quad U_i\cap W_1\cap U_{i_0}.
\end{equation} By part (i) of the lemma, we can find an open neighborhood $V\subseteq W_1\cap U_{i_0}$ of $\Gamma_1\cap \Gamma_2$ and a section $\widetilde g\in\widehat{\mathcal O}^{\mathrm{GCom\,}A}(X,Z)$ with $\widetilde g=g$ on $V$ and, hence, by \eqref{19.9.17'},
\begin{equation}\label{31.8.17neu}
\widetilde g=f_i^{-1}f^{}_{ii_0}\quad \text{on}\quad U_i\cap V.
\end{equation}
Choose an open neighborhood $W'_1\subseteq W^{}_1$ of $\Gamma_1$ and an open neighborhood $U'_{i_0}\subseteq U^{}_{i_0}$ of $\Gamma_2$ such that
$W'_1\cap U'_{i_0}\subseteq V$. Then  $W:=W'_1\cup U'_{i_0}$ is an open neighborhood of $\Gamma_1\cup\Gamma_2$, and, by \eqref{31.8.17neu},
\begin{equation}\label{31.8.17}
\widetilde g=f_i^{-1}f^{}_{ii_0}\quad \text{on}\quad U_i\cap W'_1\cap U'_{i_0}.
\end{equation}
Define an open covering
 $\Cal U^*=\{U_i^*\}_{i\in I}$ of $W$  by
\[
U_i^*=\begin{cases}U^{}_i\cap W'_1&\text{ if }i\in I\setminus\{i_0\},\\
U'_{i_0}&\text{ if } i=i_0.
\end{cases}
\] This is a refinement of $\Cal U\cap W$, and $f^*:=\big\{f^{}_{ij}\big\vert_{U^*_i\cap U^*_j}\big\}_{i,j\in I}$ is a $\big(\Cal U^*, \widehat{\mathcal O}^{\mathrm{GCom\,}A}(\cdot,Z)\big)$-cocycle which is induced by $f\big\vert_W$. Set
 \[
 f^*_i=\begin{cases}f^{}_i\big\vert_{U_i^*}\quad&\text{if}\quad i\in I\setminus \{i_0\},\\
 \widetilde g^{-1}\big\vert_{U'_{i_0}}&\text{if}\quad i=i_0.
 \end{cases}
 \]Then, by \eqref{31.8.17'}   and  \eqref{31.8.17},  $f_i^*(f_j^*)^{-1}=f^*_{ij}$ on $U^*_i\cap U_j^*$.
 Hence, $f^*$ splits, which means,  by
 Proposition \ref{17.12.15--}, that $f\big\vert_W$ splits.
\end{proof}

\begin{lem}\label{28.9.17'} Assume that $X$ is irreducible\footnote{$X$ is called
{\bf irreducible} if the set of smooth points of $X$ is connected, see, e.g., \cite[Ch. 9, \S 1.2]{GR} or \cite[Ch. V, \S 4.5, Prop. $\alpha$]{L}.}, and $Z$ is a discrete and closed subset of $X$ which contains all nonsmooth points of $X$ and all points of $X$ which are not Jordan-stable for $A$.  Then $H^1\big(X,\widehat{\mathcal O}^{\mathrm{GCom\,}A}(\cdot,Z)\big)=0$.
\end{lem}
\begin{proof} Let  an $\widehat{\mathcal O}^{\mathrm{GCom\,}A}(\cdot,Z)$-cocycle on $X$ be given,
 and let $\Cal U=\{U_\alpha\}_{\alpha\in I}$ be the covering of $f$.
Let
 $\pi:\widetilde X\to X$ be the normalization of $X$ (see, e.g., \cite[Ch. VI, \S 4]{L}).   Since $X$ is irreducible, $\widetilde X$ is connected (see, e.g., \cite[Ch. VI, \S 4.2]{L}). Since $X$ is 1-dimensional, by the Puiseux theorem (see, e.g., \cite[Ch. VI, \S 4.1]{L}), $\widetilde X$ is a Riemann surface. Since $X$ is Stein  and, hence, noncompact, $\widetilde X$ is noncompact. Set $\widetilde{\mathcal U}=\big\{\pi^{-1}(U_i)\big\}_{i\in I}$. Then, by Theorem \ref{5.8.16}, we can find a sequence $(B_\mu)_{\mu\in\N}$ of compact subsets of $\widetilde X$ such that
\begin{itemize}
\item[(a)] for each $\mu\in \N$, $B_\mu$ is contained in at least one  set  of  $\widetilde{\Cal U}$;
\item[(b)] for each $\mu\in \N^*$, $(B_0\cup\ldots\cup B_{\mu-1},B_{\mu})$ is a bump in $\widetilde X$;
\item[(c)] $\widetilde X=\bigcup_{\mu\in\N}B_\mu$;
\item[(d)] for each compact set $L\subseteq \widetilde X$, there exists $N(\Gamma)\in\N$ such that $B_\mu\cap L=\emptyset$ if $\mu\ge N(\Gamma)$.
\end{itemize}

From (a) we obtain
\begin{itemize}
\item[(a')] For each $\mu\in \N$, $\pi(B_\mu)$ is contained in at least one set of  $\mathcal U$.
\end{itemize}

Set $K_j=\pi(B_0\cup\ldots\cup B_j)$.

{\bf Statement 1.} For each $j\in\N$, there exists an open neighborhood $W$ of $K_j$ such that $f\vert_{W}$ splits.

{\em Proof of Statement 1.} By (a'), for some  $\alpha_0\in I$, $ U_{\alpha_0}$ is a neighborhood of $K_0$. Since $f\vert_{U_{\alpha_0}}$ splits (Remark \ref{30.9.17}), this proves the claim of the statement for $j=0$.

Proceeding by induction, assume that, for some $\ell\in\N$, we already have an open neighborhood $W$ of $K_\ell$ such that $f\vert_{W}$ splits.
By (b) and Lemma \ref{1.9.17}, then we can find a $Z$-adapted pair $(\Gamma_1,\Gamma_2)$ in $X$ such that
\begin{align}
&\label{30.9.17'} \Gamma_1\subseteq K_\ell,\\
&\label{1.10.17a}\Gamma_2\subseteq \pi(B_{\ell+1}),\\
&\label{30.9.17''}\Gamma_1\cup \Gamma_2=K_{\ell+1}.
\end{align}
By \eqref{30.9.17'}, $W$ is an open neighborhood of $\Gamma_1$. By \eqref{1.10.17a} and (a'), $\Gamma_2$ is contained in at least one set of $\mathcal U$. Therefore, by  Lemma \ref{27.8.17'} (iii), we can find an open neighborhood $V$ of $\Gamma_1\cup\Gamma_2$ such that $f\vert_{V}$ splits. By
 \eqref{30.9.17''}, this proves the claim of the statement for $j=\ell+1$. Statement 1 is proved.

Let $\mathrm{int\,} K_j$ be the interior of  $K_j$. For each $j\in\N$, take a relatively compact open subset $\Omega_j$ of $X$ such that $K_j\subseteq \Omega_j$. Then, by (d) ($\pi$ is proper), for each $j\in\N$, we can find $k(j)>j$ in $\N^*$ such that
\begin{equation}\label{2.2.18} \pi^{-1}(\Omega_j) \cap B_\mu=\emptyset\quad\text{if}\quad\mu\ge k(j)
\end{equation}and, hence,
\begin{equation}\label{2.2.18'} K_j \cap \pi(B_\mu)=\emptyset\quad\text{if}\quad\mu\ge k(j).
\end{equation}Moreover, by (c),  it follows from \eqref{2.2.18} that
$\pi^{-1}(\Omega_j)\subseteq B_0\cup\ldots\cup B_{k(j)}$ and, therefore, $\Omega_j\subseteq K_{k(j)}$, which implies that
\begin{equation}\label{1.2.18-}
K_j\subseteq\mathrm{int\,}K_{k(j)}.
\end{equation}
Now we  define a strictly  increasing sequence $m(j)\in \N$, $j\in \N$, setting $m(0)=0$ and $m(j+1)=k(m(j))$ for $j\in \N^*$. Then by \eqref{2.2.18'} and \eqref{1.2.18-}, for all $j\in \N$,
\begin{align}&\label{1.10.17a'}K_{m(j)}\cap  \pi(B_\mu)=\emptyset\quad\text{if}\quad \mu\ge m(j+1),\\
&\label{1.10.17a''}
K_{m(j)}\subseteq \mathrm{int\,} K_{m(j+1)}.
\end{align}

{\bf Statement 2.} Let $j,\ell\in\N$ with $m(j+1)\le \ell$, let $W$  be an open neighborhood of $K_{\ell}$ and $g\in\widehat{\mathcal O}^{\mathrm{GCom\,}A}(W,Z)$. Then there exist an open neighborhood $W'$ of $K_{\ell+1}$ and  $h\in \widehat{\mathcal O}^{\mathrm{GCom\,}A}(W',Z)$ such that
\begin{equation}\label{2.10.17'neu}
h=g\quad\text{on}\quad K_{m(j)}.
\end{equation}

{\em Proof of Statement 2.} By \eqref{1.10.17a'}, $K_{m(j)}\cap\pi(B_{\ell+1})=\emptyset$. Therefore, we can find an open  set  $V$ with
\begin{align}&\label{2.10.17'''} \pi(B_{\ell+1})\subseteq V,\\
&\label{2.10.17''}
K_{m(j)}\cap V=\emptyset.
\end{align}
By (b) and Lemma \ref{1.9.17}, we can find a $Z$-adapted pair $(\Gamma_1,\Gamma_2)$ in $X$ such that
\begin{align}
&\label{30.9.17'2} \Gamma_1\subseteq K_\ell,\\
&\label{1.10.17a2}\Gamma_2\subseteq \pi(B_{\ell+1}),\\
&\label{30.9.17''2}\Gamma_1\cup \Gamma_2=K_{\ell+1}.
\end{align}
Then: By \eqref{30.9.17'2}, $W$ is an open neighborhood of $\Gamma_1$.  By \eqref{1.10.17a2} and (a'), $\Gamma_2$ is contained in at least one set of $\mathcal U$. By \eqref{1.10.17a2} and  \eqref{2.10.17'''}, $V$ is a neighborhood of $\Gamma_2$.
Therefore, from Lemma \ref{27.8.17'} (ii) we get an open neighborhood $W'$ of $\Gamma_1\cup\Gamma_2$, i.e., by \eqref{30.9.17''2}, of $K_{\ell+1}$, a section $h\in \widehat{\mathcal O}^{\mathrm{GCom\,}A}(W',Z)$ and an open neighborhood $W_1$ of $\Gamma_1$ such that $W_1\subseteq W$ and
\begin{equation}\label{2.2.18''}
h=g\quad\text{on}\quad (W_1\cap W')\setminus V.
\end{equation}Since $\ell\ge m(j+1)\ge m(j)$ and by \eqref{30.9.17''2}, we have $K_{m(j)}\subseteq K_{\ell+1}=\Gamma_1\cup\Gamma_2$, and, by \eqref{1.10.17a2} and \eqref{2.10.17'''}, we have $\Gamma_2\subseteq V$. This implies that $K_{m(j)}\subseteq \Gamma_1\setminus V\subseteq(W_1\cap W')\setminus V$. Therefore \eqref{2.10.17'neu} follows from \eqref{2.2.18''}. Statement 2 is proved.

Applying finally often Statement 2, we obtain

{\bf Statement 3.} Let $j\in\N$, let $W$  be an open neighborhood of $K_{m(j+1)}$ and $g\in\widehat{\mathcal O}^{\mathrm{GCom\,}A}(W,Z)$. Then there exist an open neighborhood $W'$ of $K_{m(j+2)}$ and  $h\in \widehat{\mathcal O}^{\mathrm{GCom\,}A}(W,Z)$ such that
$g=h\quad\text{on}\quad K_{m(j)}$.

To prove the lemma, we have to find a family $f=\{f_\alpha\}_{\alpha\in I}$ of sections $f_\alpha\in \widehat{\mathcal O}^{\mathrm{GCom\,}A}(U_\alpha,Z)$ with $f_{\alpha\beta}^{}=f_\alpha^{} f_\beta^{-1}$ on $U_\alpha\cap U_\beta$.
Since, by (c) and \eqref{1.10.17a''}, $\bigcup_{j=1}^\infty \mathrm{int\,} K_{m(j)}=X$, for that it is sufficient to construct a sequence of families $\big\{f^{(j)}_\alpha\big\}_{\alpha\in I}$, $j\in\N$, of sections $f_\alpha^{(j)}\in \widehat{\mathcal O}^{\mathrm{GCom\,}A}(\mathrm{int\,} K_{m(j)}\cap U_\alpha,Z)$ such that, for each $j\in\N$,
\begin{itemize}
\item[S($j$):] $f_{\alpha\beta}^{}=f_\alpha^{(j)} (f_\beta^{(j)})^{-1}$ on $\mathrm{int\,} K_{m(j)}\cap U_\alpha\cap U_\beta$,  for all  $\alpha,\beta\in I$.
\item[T($j$):]  If $j\ge 2$, then $f_\alpha^{(j)}=f_\alpha^{(j-1)}$ on $K_{m(j-2)}\cap U_\alpha$,  for all $\alpha\in I$.
\end{itemize}

We do this by induction.
 By Statement 1 there is an open neighborhood  $W$ of $K_{m(0)}$ such that $f\vert_W$ splits, i.e., there is a family $\{f^{(W)}_\alpha\}_{\alpha\in I}$ of sections $f^{(W)}_\alpha\in \widehat{\mathcal O}^{\mathrm{GCom\,}A}(W\cap U_\alpha,Z)$ such that $f^{}_{\alpha\beta}=f^{(W)}_\alpha (f^{(W)}_\beta)^{-1}$ on $W\cap U_\alpha\cap U_\beta$ for all $\alpha,\beta\in I$. Setting $f^{(0)}_\alpha=f^{(W)}_\alpha\big\vert_{\mathrm{int\,} K_{m(0)}\cap U_\alpha}$,  we get a family $\big\{f^{(0)}_\alpha\big\}_{\alpha\in I}$ of sections $f_\alpha^{(0)}\in \widehat{\mathcal O}^{\mathrm{GCom\,}A}(\mathrm{int\,} K_{m(0)}\cap U_\alpha,Z)$ satisfying  S($0$). T($0$) is trivial.

Assume now, for some $\ell\in\N$, we already have families $\big\{f^{(j)}_\alpha\big\}_{\alpha\in I}$, $0\le j\le \ell$, of sections $f_\alpha^{(j)}\in \widehat{\mathcal O}^{\mathrm{GCom\,}A}(\mathrm{int\,} K_{m(j)}\cap U_\alpha,Z)$
satisfying S($j$) and T($j$) for $1\le j\le  \ell$.
We have to find a family $f^{(\ell+1)}_\alpha\in \widehat{\mathcal O}^{\mathrm{GCom\,}A}(\mathrm{int\,} K_{m(\ell+1)}\cap U_\alpha,Z)$, $\alpha\in I$, satisfying S($\ell+1$) and T($\ell+1$).

By Statement 1, we  have  $F_\alpha^{}\in \widehat{\mathcal O}^{\mathrm{GCom\,}A}(\mathrm{int\,} K_{m(\ell+1)}\cap U_\alpha,Z)$, $\alpha\in I$, such that
\begin{equation}\label{3.10.17}
f_{\alpha\beta}^{}=F_\alpha^{}F_\beta^{-1}\quad\text{on}\quad \mathrm{int\,} K_{m(\ell+1)}\cap U_\alpha\cap U_\beta.
\end{equation}
If $\ell=0$, then $f^{1}_\alpha:=F_\alpha^{}$ is as required, as S($1$) holds by \eqref{3.10.17} and T($1$) is trivial.

Let $\ell\ge 1$.
Then, by \eqref{3.10.17} and  S($\ell$), on $ \mathrm{int\,} K_{m(\ell)}\cap U_\alpha\cap U_\beta$,
\[
F_\alpha^{}F_\beta^{-1}=f^{}_{\alpha\beta}=f_\alpha^{(\ell)} (f_\beta^{(\ell)})^{-1}
\quad\text{and, therefore,}\quad
(f_\alpha^{(\ell)})^{-1}F_\alpha^{}= (f_\beta^{(\ell)})^{-1}F_\beta^{}.
\]
Hence, there exists  $g\in \widehat{\mathcal O}^{\mathrm{GCom\,}A}(\mathrm{int\,} K_{m(\ell)},Z)$ such that
\begin{equation}\label{3.10.17'}
g=(f_\alpha^{(\ell)})^{-1}F_\alpha^{}\quad\text{on}\quad \mathrm{int\,} K_{m(\ell)}\cap U_\alpha.
\end{equation} By Statement 3 (with $j=\ell-1$), we can find  $h\in \widehat{\mathcal O}^{\mathrm{GCom\,}A}(\mathrm{int\,} K_{m(\ell+1)},Z)$ with
\begin{equation}\label{3.10.17''}
h=g\quad\text{on}\quad \mathrm{int\,} K_{m(\ell-1)}.
\end{equation}

Set $f_\alpha^{(\ell+1)}=F^{}_\alpha h_{}^{-1}$ on $\mathrm{int\,} K_{m(\ell+1)}^{}\cap U_\alpha^{}$.

Then, by \eqref{3.10.17}, we have, on $\mathrm{int\,} K_{m(\ell+1)}\cap U_\alpha\cap U_\beta$, \[f_\alpha^{(\ell+1)}(f_\beta^{(\ell+1)})^{-1}=F_\alpha^{}h_{}^{-1}hF_\beta^{-1}=F_\alpha^{}F_\beta^{-1}=f^{}_{\alpha\beta},\]  i.e., S($\ell+1$) holds. Moreover, by \eqref{3.10.17''} and \eqref{3.10.17'}, we have, on $\mathrm{int\,} K_{m(\ell-1)}\cap U_\alpha$, \[f_\alpha^{(\ell+1)}=F^{}_\alpha h_{}^{-1}=F^{}_\alpha g_{}^{-1}=F^{}_\alpha(F_\alpha^{})^{-1}f_\alpha^{(\ell)}=f_\alpha^{(\ell)},\]  i.e., also T($\ell+1$) is satisfied.
\end{proof}

{\bf Proof of Theorem \ref{27.9.17''}.} Let $Z$ be the union of the set of nonsmooth points of $X$ and the set of points in $X$ which are not Jordan stable for $A$ (Def. \ref{8.6.16n'}). By Proposition \ref{11.6.16}, $Z$ is discrete and closed in $X$. Therefore (see Remark \ref{27.9.17+}) it is sufficient to prove  that $H^1\big(X,\widehat{\mathcal O}^{\mathrm{GCom\,}A}(\cdot,Z)\big)=0$.

 Let an open covering $\mathcal U=\{U_i\}_{i\in I}$ of $X$ and a $\big(\Cal U,\widehat{\mathcal O}^{\mathrm{GCom\,}A}(\cdot,Z)\big)$-cocycle  $f=\{f_{ij}\}_{j\in I}$ be given. We have to find a family $f_i\in\widehat{\mathcal O}^{\mathrm{GCom\,}A}(U_i,Z)$, $i\in I$, such that, for all $i,j\in I$  and $\zeta\in U_i\cap U_j$,
 \begin{equation}\label{29.9.17*}
 f^{}_{ij}(\zeta)=f_i^{}(\zeta)(f^{}_j(\zeta))^{-1}.
 \end{equation}

 Let $\mathfrak X$ be the family of irreducible components of $X$ (see, e.g., \cite[Ch.9, \S  2.2]{GR} or \cite[Ch.V, \S 4.5]{L}). Then, for each $Y\in\mathfrak X$, $Z\cap Y$ is discrete and closed in $Y$, contains all nonsmooth points of $Y$ and all points of $Y$ which are not Jordan stable for $A\vert_Y$. Therefore, by
 Lemma \ref{28.9.17'}, for each $Y\in \mathfrak X$, we can find a family $f^{Y}_i\in \widehat{\mathcal O}^{\mathrm{GCom\,}A\vert_Y}(U_i\cap Y,Z\cap Y)$ such that
\begin{equation}\label{28.9.17''}
f^{}_{ij}=f^Y_i(f^Y_j)^{-1}\quad \text{on}\quad Y\cap U_i\cap U_j.
\end{equation}
Since $X\setminus Z$ is smooth  and, therefore, the family $\{Y\setminus Z\}_{Y\in \mathfrak X}$ is pairwise disjoint,  for each fixed $i\in I$, the family  $\{Y\cap (U_i\setminus Z)\}_{Y\in\mathfrak X}$ is a pairwise disjoint open covering of $U_i\setminus Z$. Hence, for each $i\in I$, there is a well defined map $f_i:U_i\to \mathrm{GL}(n,\C)$ such that $f_i\big\vert_{U_i\setminus Z}\in \widehat{\mathcal O}^{\mathrm{GCom\,}A\vert_Y}(U_i\setminus Z)$ and
\[
f_i^{}(\zeta)=\begin{cases}f_i^Y(\zeta)\quad&\text{if}\quad\zeta\in Y\cap (U_i\setminus Z),\quad \text{for all }Y\in \mathfrak X,\\
I&\text{if}\quad \zeta\in Z.\end{cases}
\] Since  each $f_i^Y$ is equal to $I$ in a $Y$-neighborhood of $Z\cap Y$ and since $\mathfrak X$ is locally finite, we see that
$f_i\in \widehat{\mathcal O}^{\mathrm{GCom\,}A\vert_Y}(U_i,Z)$.  
\eqref{29.9.17*} follows from \eqref{28.9.17''} if $\zeta\in (U_i\cap U_j)\setminus Z$, and from the fact that $f_i^{}(\zeta)=f_j(\zeta)=f_{ij}(\zeta)=I$ if $\zeta\in (U_i\cap U_j)\cap Z$.
This completes the proof of Theorem \ref{27.9.17''} and, hence, of Theorem \ref{18.10.16}.

\end{document}